\documentclass{amsart}
\usepackage{graphicx}
\usepackage{amsmath}
\usepackage{amsfonts}
\usepackage{amssymb}

\newtheorem{theorem}{Theorem}

\newtheorem{corollary}[theorem]{Corollary}

\newtheorem{lemma}[theorem]{Lemma}

\newtheorem{proposition}[theorem]{Proposition}
\newtheorem{remark}[theorem]{Remark}

\begin{document}
\title[Sobolev Anisotropic inequalities]{Sobolev Anisotropic inequalities\\
with monomial weights}
\author{F. Feo}
\address{Dipartimento di Ingegneria\\
Universit\`{a} degli Studi di Napoli \textquotedblleft Parthenope
\textquotedblright\\
Centro Direzionale, Isola C4, 80143 Napoli, Italy}
\email{filomena.feo@uniparthenope.it}
\author{J. Martin}
\curraddr{Universitat Aut\`onoma de Barcelona, Department of Mathematics, Bellaterra
(Barcelona), Spain}
\email{jmartin@mat.uab.cat}
\author{M. R. Posteraro}
\address{Universit\`{a} degli Studi di Napoli \textquotedblleft Federico
II\textquotedblright\\
Dipartimento di Matematica \textquotedblleft R. Caccioppoli\textquotedblright%
\\
Complesso Monte S. Angelo, Napoli, Italy} \email{posterar@unina.it}

\begin{abstract}
We derive some anisotropic Sobolev inequalities in $\mathbb{R}^{n}$ with a
monomial weight in the general setting of rearrangement invariant spaces.
Our starting point is to obtain an integral oscillation inequality in
multiplicative form.
\end{abstract}
\maketitle

\numberwithin{equation}{section} \numberwithin{theorem}{section}

\section{Introduction}

The study of functional and geometric inequalities with monomial weights,
i.e. weights defined by
\begin{equation}
d\mu (x):=x^{A}dx=|x_{1}|^{A_{1}}\cdots |x_{n}|^{A_{n}}dx.  \label{mu}
\end{equation}
where $A=(A_{1},A_{2},\dots ,A_{n})$ is a vector in $\mathbb{R}^{n}$ with $%
A_{i}\geq 0$ for $i=1,\dots ,n,$ have been considered extensively recently
(see for example \cite{C}, \cite{CR},
\cite{BGM}, \cite{BCM}, \cite{Ng} and the references quoted therein). The
interest for this kind of problems appears when Cabr\'{e} and Ros-Oton
(motivated by an open question raised by Haim Brezis \cite{Br0},\cite{Br1})
studied in \cite{CR1} the problem of the regularity of stable solutions to
reaction-diffusion problems of double revolution in $\mathbb{R}^{2}.$ A
function $u$ has symmetry of double revolution if $u(x,y)=u(|x|,|y|)$, with $%
(x,y)\in \mathbb{R}^{D}=\mathbb{R}^{A_{1}+1}\times \mathbb{R}^{A_{2}+1}$ ($%
A_{i}$ are positive integers), \textit{i.e.} the function $u$ can be seen as
a suitable function in $\mathbb{R}^{2}$, and it is here where the Jacobian $%
|x_{1}|^{A_{1}}|x_{2}|^{A_{2}}$ appears (see \cite{CR1} for the details). In
\cite{CR}, the authors established a sharp isoperimetric inequality in $%
\left( \mathbb{R}^{n},\mu \right) $ (see also \cite{BCM}) which allows them
to obtain the following weighted Sobolev inequality.

\begin{theorem}
\textrm{\emph{(}\cite[Theorem 1.3]{CR}\emph{)}} \label{Theorem:Sobolev W1p}
Let $\mu $ be defined in (\ref{mu}), let
\begin{equation}
D=n+A_{1}+\dots +A_{n}  \label{D}
\end{equation}
and $1\leq p<D$. Then for any $f\in C_{c}^{1}(\mathbb{R}^{n})$ we have%
\footnote{$C_{c}^{1}(\mathbb{R}^{n})$ denotes the space of $C^{1}$ functions
with compact support in $\mathbb{R}^{n}.$}
\begin{equation}
\left( \int_{{\mathbb{R}}^{n}}|f|^{p^{\ast }}d\mu \right) ^{1/p^{\ast }}\leq
C\left( \int_{{\mathbb{R}}^{n}}|\nabla f|^{p}d\mu \right) ^{1/p}
\label{ineq Cabre}
\end{equation}
for some positive constant $C$, where
\begin{equation}
p^{\ast }=\frac{Dp}{D-p}.  \label{p star}
\end{equation}
\end{theorem}

As in the unweighted case a scaling argument shows that the exponent $%
p^{\ast }$ is optimal, in the sense that (\ref{p star}) can not hold with
any other exponent. Moreover the exponent $p^{\ast }$ is exactly the same as
in the classical Sobolev inequality, but in this case the ''dimension'' is
given by $D$ (instead of $n$). If $A_{1}=...=A_{n}=0$, then exponent $%
p^{\ast }$ and inequality (\ref{ineq Cabre}) are exactly the classical ones.

We observe that when $p>1$ and $Ai<p-1$ for all $i=1,\cdots ,n$ the weight
in \eqref{mu} belongs to the Muckenhoupt class $A_{p}$, i.e. but, in general
the monomial weight does not satisfy the Muckenhoupt condition.

The main purpose of this paper is to obtain some anisotropic Sobolev
inequalities on $\mathbb{R}^{n}$ with monomial weight $x^{A}$ in the general
setting of rearrangement invariant spaces (\textit{e.g.} $L^{p}$, Lorentz,
Orlicz, Lorentz-Zygmund, etc...). To this end, we will use the
''symmetrization by truncation principle'', developed by Milman-Mart\'{i}n
in \cite{MMAdv} (see also \cite{MMpoten} and \cite{MMP07}). This method will
provide us a family of rearrangement pointwise inequalities between the
special difference\footnote{$f_{\mu }^{\ast }$ is the decreasing
rearrangement of $f$ with respect the measure $\mu $, and\ $f_{\mu }^{\ast
\ast }(t)=\frac{1}{t}\int_{0}^{t}f_{\mu }^{\ast }(s)ds$ (see Subsection 2.1).%
} $O_{\mu }(f,t):=f_{\mu }^{\ast \ast }(t)-f_{\mu }^{\ast }(t)$ (called the
oscillation of $f$) and the product of the rearrangements of the partial
derivative of $f$ (see Theorem \ref{main th} below) that will be the key to
obtain anisotropic inequalities. More precisely we will prove that
inequality (\ref{ineq Cabre}) with $p=1$ is equivalent to the following
oscillation inequality:
\begin{equation}
\!\!\int_{0}^{t}\!\left( O_{\mu }(f,\cdot )\left( \cdot \right) ^{-\frac{1}{D%
}}\,\right) ^{\ast }(s)ds\preceq \!\int_{0}^{t}\prod_{i=1}^{n}\left[ \left(
\frac{d}{ds}\int_{\{|f|>f_{\mu }^{\ast }(s)\}}\left| f_{x_{i}}\right| d\mu
\right) ^{\ast }(\tau )\right] ^{\frac{A_{i}+1}{D}}d\tau \!\!  \label{pepe}
\end{equation}
for every $t>0$. The rearrangements without subscript $\mu $ are
rearrangement with respect to Lebesgue measure on $\left( 0,\infty \right) $%
, $f_{x_{i}}=\frac{\partial f}{\partial x_{i}}$ and symbol $f\preceq g$
means that there exists an universal constant $c$ (independent of all
parameters involved) such that $f\leq cg$.

Inequality (\ref{pepe}) contains the basic information to obtain anisotropic
Sobolev inequalities on rearrangement invariant spaces, since given a
rearrangement invariant space $X$ on $\left( \mathbb{R}^{n},\mu \right) $,
Hardy's inequality (see (\ref{Hardy}) below) implies\footnote{%
The spaces $\bar{X}$ are defined in Section \ref{secc:ri} below.}
\begin{equation}
\left\| O_{\mu }(f,t)t^{-\frac{1}{D}}\right\| _{\bar{X}}\preceq \left\|
\prod_{i=1}^{n}\left[ \left( \frac{d}{ds}\int_{\{|f|>f_{\mu }^{\ast
}(s)\}}\left| f_{x_{i}}\right| d\mu \right) ^{\ast }(\tau )\right] ^{\frac{%
A_{i}+1}{D}}(t)\right\| _{\bar{X}}. \label{esto}
\end{equation}
For example, given $p_{1},\cdots ,p_{n}\geq 1,$ let $\overline{p}$ be the
weighted harmonic mean between $p_{1},\cdots ,p_{n}$, \textit{i.e.} \
\begin{equation}
\frac{1}{\overline{p}}=\frac{1}{D}\sum_{i=1}^{n}\frac{A_{i}+1}{p_{i}},
\label{harm}
\end{equation}
then (\ref{esto}) implies (see Theorem 4.3 below)
\begin{equation}
\Vert O_{\mu }(f,t)t^{-\frac{1}{D}}\Vert _{\bar{X}^{(\overline{p})}}\preceq
\prod_{i=1}^{n}\left\| f_{x_{i}}\right\| _{X^{(p_{i})}}^{\frac{A_{i}+1}{D}},
\label{ej3}
\end{equation}
where $X^{(p)}=\{f:\left| f\right| ^{p}\in X\}$ endowed with the norm $%
\left\| f\right\| _{X^{(p)}}=\left\| \left| f\right| ^{p}\right\| _{X}^{1/p}.
$

In the particular case that $X=L^{1}$ and $\overline{p}<D,$ then (\ref{ej3})
becomes (see Proposition \ref{prop L1} below)
\begin{equation*}
\Vert f\Vert _{L_{\mu }^{{\overline{p}}^{\ast }}}\preceq
\prod_{i=1}^{n}\left\| f_{x_{i}}\right\| _{L_{\mu }^{p_{i}}}^{\frac{A_{i}+1}{%
D}}\text{ \ \ }\forall f\in C_{c}^{1}(\mathbb{R}^{n}).  
\end{equation*}
where $\bar{p}^{\ast }=\frac{D\bar{p}}{D-\bar{p}}.$

In particular if $p=p_{1}=\cdots =p_{n},$ then $\overline{p}=p$, $\bar{p}%
^{\ast }=p^{\ast }$, and we get
\begin{equation}
\Vert f\Vert _{L_{\mu }^{{\overline{p}}^{\ast }}}\preceq
\prod_{i=1}^{n}\left\| f_{x_{i}}\right\| _{L_{\mu }^{p}}^{\frac{A_{i}+1}{D}}%
\text{ \ \ }\forall f\in C_{c}^{1}(\mathbb{R}^{n}),  \label{ej2}
\end{equation}
which implies (\ref{ineq Cabre}).

In the unweighted case, i.e. $A_{1}=\cdots ,A_{n}=0$, inequality
(\ref{ej2}) is well-known (see \textit{e.g.} \cite{Tr}, \cite{Tar}
and \cite{KP}).
Anisotropic inequalities involving Orlicz norm defined using an $n$%
-dimensional Young function were also studied in \cite{C_fully}. However, as
far we know, our anisotropic inequalities involving rearrangement invariant
spaces are new in this context.

The paper is organized as follows. In Section \ref{preli} we provide a brief
review on the rearrangements of functions and the theory of rearrangement
invariant spaces. In Section \ref{Main} we will prove our main result
(Theorem \ref{main th} below) which establishes the equivalence between (\ref
{ineq Cabre}) with $p=1$ and (\ref{pepe}). Finally in Section \ref{applica}
we use the oscillation inequality \eqref{pepe} to derive anisotropic Sobolev
inequalities in $\mathbb{R}^{n}$ with monomial weight $x^{A}$ in the general
setting of rearrangement invariant spaces, with special attention in the
case of Lebesgue spaces, Lorentz spaces, Lorentz-Zygmund spaces, Gamma
spaces and the recent class of $G\Gamma$ spaces.

\section{Notations and preliminary results\label{preli}}

We briefly recall the basic definitions of rearrangements and of
rearrangement-invariant (r.i.) spaces referring the reader to \cite{BS} and
\cite{KPS}.

\subsection{Rearrangement of functions}

Let $\mu $ an absolutely continuous measure with respect to Lebesgue measure
on $\mathbb{R}^{n}$. For a $\mu $-measurable function $u:\mathbb{R}%
^{n}\rightarrow \mathbb{R},$ the distribution function of $u$ is given by
\begin{equation*}
\mu _{u}(s)=\mu \{x\in {\mathbb{R}^{n}}:\left| u(x)\right| >s\}\text{ \ \ \
\ }s\geq 0.
\end{equation*}
The \textbf{decreasing rearrangement} $u_{\mu }^{\ast }$ of $u$ is the
right-continuous non-increasing function from $[0,\infty )$ into $[0,\infty
] $ which is equimeasurable with $u$. Namely,
\begin{equation*}
u_{\mu }^{\ast }(t)=\inf \{s\geq 0:\mu _{u}(s)\leq t\}{\text{ \ \ \ \ }t\geq
0}.
\end{equation*}

We also define $u_{\mu }^{\ast \ast }:\left( 0,\infty \right) \rightarrow
\left( 0,\infty \right) $ as
\begin{equation*}
u_{\mu }^{\ast \ast }(t)=\frac{1}{t}\int_{0}^{t}u_{\mu }^{\ast }(s)ds,
\end{equation*}
Note that $u_{\mu }^{\ast \ast }$ is also decreasing and $u_{\mu }^{\ast
}\leq u_{\mu }^{\ast \ast },$ moreover
\begin{equation*}
(u+v)_{\mu }^{\ast \ast }(t)\leq u_{\mu }^{\ast \ast }(t)+v_{\mu }^{\ast
\ast }(t)
\end{equation*}
for $t>0.$

The \textbf{oscillation }of $u$ is defined by
\begin{equation*}
O_{\mu }(u,t):=u_{\mu }^{\ast \ast }(t)-u_{\mu }^{\ast }(t).
\end{equation*}
Note that 
\begin{equation}
tO_{\mu }(u,t)=\int_{u_{\mu }^{\ast }(t)}^{\infty }\mu _{u}(s)ds
\label{crece}
\end{equation}
is increasing.

When rearrangements are taken with respect to the Lebesgue measure on $%
(0,\infty )$, we may omit the measure and simply write $u^{\ast }$ and $%
u^{\ast \ast }$, etc...

\subsection{Rearrangement invariant spaces\label{secc:ri}}

We say that a Banach function space $X=X({\mathbb{R}}^{n})$ on $(\mathbb{R}%
^{n},\mu )$ is \textbf{rearrangement-invariant (r.i.) space}, if $g\in X$
implies that $f\in X$ for all $\mu -$measurable functions $f$ such that $%
f_{\mu }^{\ast }=g_{\mu }^{\ast },$ and $\Vert f\Vert _{X}=\Vert g\Vert _{X}$%
.

A basic property of rearrangements is the Hardy-Littlewood inequality which
tells us that, if $u$ and $w$ are two $\mu $-measurable functions on $%
\mathbb{R}^{n}$ , then
\begin{equation}
\int_{\mathbb{R}^{n}}|u(x)w(x)|\,d\mu \leq \int_{0}^{\infty }u_{\mu }^{\ast
}(t)w_{\mu }^{\ast }(t)\,dt.  \label{HL}
\end{equation}
An important consequence (\ref{HL}) is the Hardy-Littlewood-P\'{o}lya
principle stating that
\begin{equation}
\int_{0}^{r}f_{\mu }^{\ast }(s)ds\leq \int_{0}^{r}g_{\mu }^{\ast }(s)ds\quad
\forall r>0\Rightarrow \left\| f\right\| _{X}\leq \left\| g\right\| _{X}%
\text{ \ for any r.i. space }X.  \label{Hardy}
\end{equation}

A r.i. space $X({\mathbb{R}}^{n})$ can be represented by a r.i. space on $%
(0,+\infty ),$ with Lebesgue measure, $\bar{X}=\bar{X}(0,\infty ),$ such
that
\begin{equation*}
\Vert f\Vert _{X}=\Vert f_{\mu }^{\ast }\Vert _{\bar{X}},
\end{equation*}
for every $f\in X.$ A characterization of the norm $\Vert \cdot \Vert _{\bar{%
X}}$ is available (see \cite[Theorem 4.10 and subsequent remarks]{BS}). The
space $\bar{X}$ is called the \textbf{representation space} of $X$.

If $X$ is a r.i space, we have
\begin{equation*}
L_{\mu }^{1}\cap L^{\infty }\subset \bar{X}\subset L_{\mu }^{1}+L^{\infty },
\end{equation*}
with continuous embeddings.

The \textbf{associate space} $X^{\prime }$ of $X$ is the r.i. space of all
measurable functions $h$ for which the r.i. norm given by
\begin{equation*}
\left\| h\right\| _{X^{\prime }}=\sup_{g\neq 0}\frac{\int_{{\mathbb{R}}%
^{n}}\left| g(x)h(x)\right| d\mu (x)}{\left\| g\right\| _{X}}=\sup_{g\neq 0}%
\frac{\int_{0}^{\infty }h_{\mu }^{\ast }(s)g_{\mu }^{\ast }(s)ds}{\left\|
g\right\| _{X}}.
\end{equation*}
In particular the following generalized H\"{o}lder inequality
\begin{equation*}
\int_{{\mathbb{R}}^{n}}\left| g(x)h(x)\right| d\mu (x)\leq \left\| g\right\|
_{X}\left\| h\right\| _{X^{\prime }}
\end{equation*}
holds.

Classically conditions on r.i. spaces are given in terms of the Hardy
operators defined by
\begin{equation*}
Pf(t)=\frac{1}{t}\int_{0}^{t}f(s)ds;\text{ \ \ \ }Q_{a}f(t)=\frac{1}{t^{a}}%
\int_{t}^{\infty }s^{a}f(s)\frac{ds}{s},\text{ \ \ }0\leq a<1,
\end{equation*}
(if $a=0,$ we shall write $Q$ instead of $Q_{0})$.

The boundedness of these operators on r.i. spaces can be described in terms
of the so called \textbf{Boyd indices}\footnote{%
Introduced by D.W. Boyd in \cite{boyd}.} defined by
\begin{equation*}
\bar{\alpha}_{X}=\inf\limits_{s>1}\dfrac{\ln h_{X}(s)}{\ln s}\text{ \ \ and
\ \ }\underline{\alpha }_{X}=\sup\limits_{s<1}\dfrac{\ln h_{X}(s)}{\ln s},
\end{equation*}
where $h_{X}(s)$ denotes the norm of the compression/dilation operator $%
E_{s} $ on $\bar{X}$, defined for $s>0,$ by $E_{s}f(t)=f^{\ast }(\frac{t}{s}%
) $. For example if $X=L_{\mu }^{p}$ with $p>1$, then $\bar{\alpha}_{X}=%
\underline{\alpha }_{X}=\frac{1}{p}$. It is well known that
\begin{equation*}
\begin{array}{c}
P\text{ is bounded on }\bar{X}\text{ }\Leftrightarrow \overline{\alpha }%
_{X}<1, \\
Q\text{ is bounded on }\bar{X}\text{ }\Leftrightarrow \underline{\alpha }%
_{X}>a.
\end{array}
\end{equation*}

The next two Lemmas will be used in Section \ref{applica}.

\begin{lemma}
\emph{(}see \cite[Page 43]{KPS}\emph{)} Let $X$ be a r.i. space and let $%
0\leq \theta _{i}\leq 1$ such that {$\sum_{i=1}^{n}\theta _{i}=1{,}$ then} {%
\
\begin{equation}
\Vert \prod_{i=1}^{n}|f_{i}|^{\theta _{i}}\Vert _{X}\leq
\prod_{i=1}^{n}\Vert f_{i}\Vert _{X}^{\theta _{i}}.  \label{holder
X}
\end{equation}
}
\end{lemma}

\begin{lemma}
\label{01}Let $g,h$ be two positive measurable functions on $(0,\infty )$
such that
\begin{equation*}
g(s)\leq h^{\ast \ast }(s),\text{ for all }s\in (0,\infty ),  \label{h1}
\end{equation*}
and
\begin{equation*}
\int_{0}^{t}g(s)\leq \int_{0}^{t}h^{\ast }(s)ds,\text{ for all }t\in
(0,\infty ).  \label{h2}
\end{equation*}
Then
\begin{equation*}
\int_{0}^{t}g^{\ast }(s)ds\leq 4\int_{0}^{t}h^{\ast }(s)ds\,\text{ for all }%
t\in (0,\infty ).
\end{equation*}
Therefore for any r.i. space $X$%
\begin{equation*}
\left\| g\right\| _{X}\leq 4\left\| h\right\| _{X}.
\end{equation*}
\end{lemma}

The proof of the previous lemma is implicitly contained in the proof of
Theorem 1.2 of \cite{Cia}, a detailed proof can be found in \cite{MMpoten}.

\subsubsection{Examples}

\paragraph{\textbf{Convexifications of r.i. spaces}.}

A way to construct r.i. spaces is through the so-called p-convexification,
which is the generalization of the procedure to construct $L^{p}$ spaces, $%
1<p<\infty $, starting from $L^{1}$. If $X$ is a r.i. space, the $p-$\textbf{%
convexification} $X^{(p)}$ of $X,$ (cf. \cite{lt}) is the r.i. space defined
$X^{(p)}=\{f:\left| f\right| ^{p}\in X\}$ endowed with the following norm
\begin{equation}
\left\| f\right\| _{X^{(p)}}=\left\| \left| f\right| ^{p}\right\| _{X}^{1/p}.
\label{convessifi}
\end{equation}
The same is true for the functional $\left\| \cdot \right\|
_{X^{\left\langle p\right\rangle }}$ defined as
\begin{equation}
\left\| f\right\| _{X^{\left\langle p\right\rangle }}=\left\| \left( \left(
\left| f\right| ^{p}\right) ^{\ast \ast }\right) ^{1/p}\right\| _{X}.
\label{convessifi1}
\end{equation}
Spaces $X^{\left\langle p\right\rangle }$ have been introduced in \cite{CPL}
in connection with the study of Sobolev embeddings into
rearrangement-invariant spaces defined by a Frostman measure.

\paragraph{\textbf{The generalized Lorentz spaces $\Lambda^{p,q}(w)$}.}

Given $1\leq p,q<\infty $, and $w$ a weight (a positive locally integrable
function) on $\left( 0,\infty \right) .$ The generalized Lorentz spaces $%
\Lambda ^{p,q}(w)$ are defined by measurable functions on $\left( 0,\infty
\right) $ such that
\begin{equation}
\left\| f\right\| _{\Lambda ^{p,q}(w)}:=\left( \int_{0}^{\infty
}\left( t^{1/p}f^{\ast }(t)\right) ^{q}w(t)\frac{dt}{t}\right)
^{1/q}<\infty .\label{lornorm}
\end{equation}
If $p=q,$ we write $\Lambda ^{p}(w)$ instead $\Lambda ^{p,p}(w).$ We denote
by $W(t)=\int_{0}^{t}w(s)ds.$ It is know (see \cite{CS}) that $\Lambda
^{p,q}(w)=\Lambda ^{q}(W^{q/p-1}w)$. A weight $w$ is called a $B_{p}-$weight
if three is $C>0$ such that
\begin{equation*}
\int_{t}^{\infty }\frac{w(s)}{s^{p}}\leq \frac{C}{r^{p}}\int_{0}^{t}w(s)ds%
\text{, \ \ }t>0.
\end{equation*}
The $B_{p}$ class satisfies that $B_{r}\subset B_{p}\ $if $r\leq p$. If $%
w\in B_{p}$ then (see \cite{sw})
\begin{equation*}
\left\| f\right\| _{\Lambda ^{p}(w)}\simeq \left( \int_{0}^{\infty }f^{\ast
\ast }(t)^{p}w(t)dt\right) ^{1/p},
\end{equation*}
therefore Lorentz spaces defined by $B_{p}-$weights are r.i. spaces.
Moreover, since (see \cite[Theorem 6.5]{Ne})
\begin{equation*}
w\in B_{p}\Rightarrow W^{q/p-1}w\in B_{q}
\end{equation*}
we get that $\Lambda ^{p,q}(w)$ is a r.i. space if $w\in B_{p}.$

Given $f$ a $\mu -$measurable function on $\mathbb{R}^{n}$ we define
\begin{equation*}
\Vert f\Vert _{\Lambda _{\mu }^{{p,q}}(w)}=\left\{ f:\text{ }\Vert f_{\mu
}^{\ast }\Vert _{\Lambda ^{p.q}(w)}<\infty \right\} .
\end{equation*}

Typical examples of generalized Lorentz spaces are the $L^{p}$-spaces and
the \textbf{Lorentz spaces} $L^{p,q},\ $defined either for $p=q=1$ or $%
p=q=\infty $ or $1<p<\infty $, $1\leq q\leq \infty $ by
\begin{equation*}
\Vert f\Vert _{L_{\mu}^{p,q}}:=\left\| t^{1/p-1/q}f^{\ast }_{\mu}(t)\right\|
_{L_{[0,\infty )}^{q}}<+\infty ,
\end{equation*}
and, more generally, the \textbf{Lorentz-Zygmund spaces}, defined for $%
1<p<\infty $, $1\leq q\leq \infty $ and $\alpha \in \mathbb{R}$ by
\begin{equation*}
\Vert f\Vert _{L_{\mu}^{p,q}(\log L)^{\alpha }}:=\left\| t^{1/p-1/q}(1+\log
\left| f\right| )^{\alpha /q}f_{\mu}^{\ast }(t)\right\| _{L_{[0,\infty
)}^{q}}<+\infty
\end{equation*}

Let us notice that, in spite of the notation, the quantities $\Vert f\Vert
_{L^{p,q}}$ and $\Vert f\Vert _{L^{p,q}(\log L)^{\alpha }}$ need not be
norms; however, they can be turned into equivalent norms, when $1<p<\infty $%
, replacing $f^{\ast }$ by $f^{\ast \ast }.$

\paragraph{\textbf{The Gamma spaces $\Gamma^{p}(w)$}.}

Let $1\leq p<\infty .$ Let $w$ be an admissible weight, i.e.
\begin{equation}
\int_{0}^{t}w(s)ds<\infty \text{ and }\int_{t}^{\infty }\frac{w(s)}{s^{p}}%
ds<\infty .  \label{gweg}
\end{equation}
The Gamma space $\Gamma ^{p}(w)$ is the r.i. space defined as the set of
measurable functions such that
\begin{equation}
\left\| f\right\| _{\Gamma ^{p}(w)}:=\left( \int_{0}^{\infty
}f^{\ast \ast }(s)^{p}w(s)ds\right) ^{1/p}<\infty. \label{gammanorm}
\end{equation}
Given $f$ a $\mu -$measurable function on $\mathbb{R}^{n}$ we define
\begin{equation*}
\Vert f\Vert _{\Gamma _{\mu }^{{p}}(w)}=\left\{ f:\Vert f_{\mu }^{\ast
}\Vert _{\Gamma ^{p}(w)}<\infty \right\} .
\end{equation*}

\paragraph{\textbf{The $G\Gamma^{p}(p,m,w)$ spaces}.}

Let $1\leq p,m<\infty $ and let $w$ be a weight satisfying that
\begin{equation}
\int_{0}^{\infty }\min (s,t)^{m/p}w(t)dt<\infty ,\text{ \ \ }s>0.
\label{ggweg}
\end{equation}
The $G\Gamma (p,m,w)-$spaces are defined by
\begin{equation}
G\Gamma (p,m,w)=\left\{ f:\left\| f\right\| _{G\Gamma
(p,m,w)}=\left( \int_{0}^{\infty }\left( \int_{0}^{t}f^{\ast
}(s)^{p}ds\right) ^{m/p}w(t)dt\right) ^{1/m}<\infty \right\}
.\label{ggammanorm}
\end{equation}

These spaces\ has been introduced in \cite{FR1} in connection with
compact Sobolev type embedding results, since then its turn out to
be important and several papers devoted to the study of this spaces
have been published (see e.g. \cite{FR}, \cite{FRGKR}, \cite{GMPS}
and the references quoted therein).

\subsection{Some remarks about function spaces defined by oscillations}

In this subsection we analyze functional properties of function spaces whose
definition involves the oscillation of $f$. The principal difficulty dealing
with the functional $O_{\mu }(f,t)$ is its nonlinearity. Therefore function
spaces whose definition involves this quantity are not linear spaces.

Consider the Hardy type operator defined on positive measurable function on $%
\left( 0,\infty \right) $ by
\begin{equation*}
\bar{Q}f(t)=\int_{t}^{\infty }s^{1/D}f(s)\frac{ds}{s}.
\end{equation*}

\begin{theorem}
\label{teoind} Let $X$ be a r.i. space on $\left( 0,\infty \right) $ and let
us assume that $X$ does not contain constant functions.

\noindent i) If $\underline{\alpha }_{X}>\frac{1}{D}$, then

\begin{equation*}
\Vert t^{-1/D}f^{\ast \ast }(t)\Vert _{X}\simeq \Vert t^{-1/D}[f^{\ast \ast
}(t)-f^{\ast }(t)]\Vert _{X}.
\end{equation*}

\noindent ii) If $\bar{\alpha}_{X}<\frac{1}{D}$, then
\begin{equation*}
\Vert f\Vert _{L^{\infty }}\preceq \Vert t^{-1/D}[f^{\ast \ast }(t)-f^{\ast
}(t)]\Vert _{X}+\left\| f\right\| _{L^1_\mu+L^{\infty }}.
\end{equation*}
\end{theorem}

\begin{proof}
\textit{i}) Let $f\in X$. Since $X$ does not contain constant functions, $%
f^{\ast \ast }(\infty )=0.$ An elementary computation shows that
$\left(
-f^{\ast \ast }\right) ^{\prime }(t)=\frac{f^{\ast \ast }(t)-f^{\ast }(t)}{t}%
,$ thus by the fundamental theorem of Calculus we get
\begin{equation*}
f^{\ast \ast }(t)=\int_{t}^{\infty }\left( f^{\ast \ast }(t)-f^{\ast
}(t)\right) \frac{ds}{s}.
\end{equation*}
Therefore,
\begin{eqnarray*}
\Vert t^{-1/D}f^{\ast \ast }(t)\Vert _{X} &=&\left\|
t^{-1/D}\int_{t}^{\infty }\left( f^{\ast \ast }(s)-f^{\ast }(s)\right) \frac{%
ds}{s}\right\| _{X} \\
&=&\left\| t^{-1/D}\int_{t}^{\infty }s^{1/D}s^{-1/D}\left( f^{\ast
\ast
}(s)-f^{\ast }(s)\right) \frac{ds}{s}\right\| _{X} \\
&\preceq &\left\| s^{-1/D}\left( f^{\ast \ast }(s)-f^{\ast
}(s)\right) \right\| _{X}\text{ \ (since }\underline{\alpha
}_{X}>\frac{1}{D}).
\end{eqnarray*}
The converse inequality is obvious.

\textit{ii}) Let us assume now that
$\bar{\alpha}_{X}${$<\frac{1}{D}$}. We get
\begin{eqnarray*}
\Vert f\Vert _{L^{\infty }}+\left\| f\right\| _{L^1_\mu+L^{\infty }}
&=&f^{\ast \ast }(0)-f^{\ast \ast }(1)=\int_{0}^{1}\left( f^{\ast
\ast
}(s)-f^{\ast }(s)\right) \frac{ds}{s} \\
&=&\int_{0}^{1}s^{1/D}\left( s^{-1/D}\left( f^{\ast \ast
}(s)-f^{\ast
}(s)\right) \right) \frac{ds}{s} \\
&\leq &\Vert s^{-1/D}\left( f^{\ast \ast }(s)-f^{\ast }(s)\right)
\Vert _{X}\Vert s^{1/D-1}\chi _{\lbrack 0,1]}(s)\Vert _{X^{\prime
}}.
\end{eqnarray*}
It is enough to check that $\Vert s^{1/D-1}\chi _{\lbrack
0,1]}(s)\Vert
_{X^{\prime }}<\infty .$ Since $\bar{\alpha}_{X}=1-\underline{\alpha }%
_{X^{^{\prime }}}<\frac{1}{D},${\ we can select }
\begin{equation}
1-1/D<\beta <\underline{\alpha }_{X^{^{\prime }}}.  \label{salvado}
\end{equation}
By the definition of indices, there is $c>0$ such that
\begin{equation}
h_{X^{^{\prime }}}(2^{-k})\leq c2^{-k\beta }\quad \forall k\geq 0.
\label{salvo}
\end{equation}
For any $k\geq 0,$ write $I_{k}=[2^{-k-1},2^{-k}).$ Since $\chi
_{\lbrack 0,1]}(s)=\sum_{k=0}^{\infty }\chi _{I_{k}}(s),$ we get
\begin{align*}
\Vert s^{1/D-1}\chi _{\lbrack 0,1]}(s)\Vert _{X^{\prime }}& \leq
\sum_{k=0}^{\infty }\Vert s^{1/D-1}\chi _{I_{k}}(s)\Vert _{X^{\prime
}}\leq \sum_{k=0}^{\infty }2^{k\left( 1-1/D\right) }\Vert \chi
_{I_{k}}(s)\Vert
_{X^{\prime }} \\
& =\sum_{k=0}^{\infty }2^{k\left( 1-1/D\right) }\Vert D_{2^{-k}}\chi
_{\lbrack 1/2,1)}(s)\Vert _{X^{\prime }} \\
& \leq \sum_{k=0}^{\infty }2^{k\left( 1-1/D\right)
}h_{X^{\prime}}(2^{-k})\Vert \chi _{\lbrack 1/2,1)}(s)\Vert _{X^{\prime }} \\
& =\sum_{k=0}^{\infty }2^{k\left( 1-1/D-\beta \right) }2^{k\beta
}h_{X^{\prime}}(2^{-k})\Vert \chi _{\lbrack 1/2,1)}(s)\Vert _{X^{\prime }} \\
& \leq c\sum_{k=0}^{\infty }2^{k\left( 1-1/D-\beta \right) }\Vert
\chi _{\lbrack 1/2,1)}(s)\Vert _{X^{\prime }}<\infty \text{ \ (by
(\ref{salvo}) and (\ref{salvado}))}.
\end{align*}
\end{proof}

\section{Self-improvement of Sobolev inequality\label{Main}}

Let $W_{0}^{1,1}(\mathbb{R}^{n},\mu )$ be the closure of the space $%
C_{c}^{1}(\mathbb{R}^{n})$ under the norm
\begin{equation*}
\Vert u\Vert _{W_{0}^{1,1}(\mathbb{R}^{n},\mu )}=\int_{\mathbb{R}%
^{n}}(|\nabla u|+|u|)d\mu .
\end{equation*}
Inequality \eqref{ineq Cabre} with $p=1$ implies the following anisotropic
Sobolev inequality
\begin{equation}
\Vert f\Vert _{L_{\mu }^{\frac{D}{D-1}}}\preceq\sum_{i=1}^{n}\left\| \frac{%
\partial f}{\partial x_{i}}\right\| _{L_{\mu }^{1}}  \label{inq 1}
\end{equation}
for $f\in W_{0}^{1,1}(\mathbb{R}^{n},\mu )$.

\begin{theorem}
\label{main th} Let $\mu $ and $D$ be defined as in \eqref{mu} and \eqref{D}
respectively. Let $f\in W_{0}^{1,1}(\mathbb{R}^{n},\mu )$ the following
statements hold and are equivalent:

\begin{itemize}
\item[\textit{i})]  (Poincar\'{e} inequality)
\begin{equation}
\Vert f\Vert _{L_{\mu }^{\frac{D}{D-1}}}\preceq \sum_{i=1}^{n}\left\|
f_{x_{i}}\right\| _{L_{\mu }^{1}}.  \label{un}
\end{equation}

\item[\textit{ii})]  (Poincar\'{e} inequality in multiplicative form)
\begin{equation}
\Vert f\Vert _{L_{\mu }^{\frac{D}{D-1}}}\preceq \prod_{i=1}^{n}\left\|
f_{x_{i}}\right\| _{L_{\mu }^{1}}^{\frac{A_{i}+1}{D}}.  \label{dos}
\end{equation}

\item[\textit{iii})]  (Mazya-Talenti's inequality in multiplicative form)
The function $f_{\mu }^{\ast }$ is locally absolutely continuous and for all
$s>0$ we have that
\begin{equation}
s^{1-1/D}\left( -f_{\mu }^{\ast }\right) ^{\prime }(s)\preceq
\prod_{i=1}^{n}\left( \frac{d}{ds}\int_{\left\{ \left| f\right| >f_{\mu
}^{\ast }(s)\right\} }\left| f_{x_{i}}\right| d\mu \right) ^{\frac{A_{i}+1}{D%
}}.  \label{tres}
\end{equation}

\item[\textit{iv})]  (Oscillation inequality in multiplicative form) For all
$1\leq p<\infty $ and for all $t>0$ we get
\begin{equation}
\int_{0}^{t}\left( O_{\mu }(f,\cdot )^{p}\left( \cdot \right) ^{-\frac{p}{D}%
}\,\right) ^{\ast }(s)ds\preceq \int_{0}^{t}\prod_{i=1}^{n}\left[ \left(
\frac{d}{ds}\int_{\{|f|>f_{\mu }^{\ast }(s)\}}\left| f_{x_{i}}\right| d\mu
\right) ^{\ast }(\tau )\right] ^{p\left( \frac{A_{i}+1}{D}\right) }d\tau ,
\label{O}
\end{equation}
where rearrangements without subscript $\mu $ are taken with respect to
Lebesgue measure on $\left( 0,\infty \right) $.

\item[\textit{v})]
\begin{equation*}  
\Vert f\Vert _{L_{\mu }^{\frac{D}{D-1},1}}\preceq \prod_{i=1}^{n}\left\|
f_{x_{i}}\right\| _{L_{\mu }^{1}}^{\frac{A_{i}+1}{D}}.
\end{equation*}
\end{itemize}
\end{theorem}

\begin{proof}
$i)\Rightarrow ii)$

We follow a scaling argument as in \cite[Lemma 7]{Tar}. Indeed we apply (%
\ref{un}) to function $w(x)=f(\lambda _{1}x_{1},\cdots ,\lambda
_{n}x_{n})$ and we obtain (\ref{dos}) by choosing
\begin{equation*}
\lambda _{i}=\prod_{j\neq i}\left\| f_{x_{i}}\right\| _{L_{\mu
}^{1}}.
\end{equation*}

$ii) \Rightarrow iii)$

Since $L_{\mu }^{\frac{D}{D-1}}$ is continuously embedded in $L_{\mu }^{\frac{D%
}{D-1},\infty }$with constant $1$, inequality (\ref{dos}) implies
the weaker one
\begin{equation}
\sup_{t>0}t\left| \left\{ x\in \mathbb{R}^{n}:\left| f(x)\right|
>t\right\} \right| ^{\frac{D-1}{D}}\leq nC\prod_{i=1}^{n}\left\|
f_{x_{i}}\right\| _{L_{\mu }^{1}}^{\frac{A_{i}+1}{D}}.
\label{dosdos1}
\end{equation}
Let $0<t_{1}<t_{2}<\infty ,$ the truncations of $f$ are defined by
\begin{equation*}
f_{t_{1}}^{t_{2}}(x)=\left\{
\begin{array}{ll}
t_{2}-t_{1} & \text{if }\left| f(x)\right| >t_{2}, \\
\left| f(x)\right| -t_{1} & \text{if }t_{1}<\left| f(x)\right| \leq
t_{2},
\\
0 & \text{if }\left| f(x)\right| \leq t_{1}.
\end{array}
\right.
\end{equation*}
Observe that if $f\in W_{0}^{1,1}(\mathbb{R}^{n},\mu )$ then $%
f_{t_{1}}^{t_{2}}\in W_{0}^{1,1}(\mathbb{R}^{n},\mu ),$ therefore replacing $%
f$ by $f_{t_{1}}^{t_{2}}$ in (\ref{dosdos1}) we obtain
\begin{equation*}
\sup_{t>0}t\left| \left\{ x\in \mathbb{R}^{n}:\left|
f_{t_{1}}^{t_{2}}(x)\right| >t\right\} \right| ^{\frac{D-1}{D}}\leq
nC\prod_{i=1}^{n}\left( \int_{\mathbb{R}^{n}}\left| \frac{\partial
f_{t_{1}}^{t_{2}}}{\partial x_{i}}\right| d\mu \right)
^{\frac{A_{i}+1}{D}}.
\end{equation*}
Since
\begin{equation*}
\sup_{t>0}t\left| \left\{ x\in \mathbb{R}^{n}:\left|
f_{t_{1}}^{t_{2}}(x)\right| >t\right\} \right| ^{\frac{D-1}{D}}\geq
\left( t_{2}-t_{1}\right) \left| \left\{ x\in \mathbb{R}^{n}:\left|
f(x)\right| \geq t_{2}\right\} \right| ^{\frac{D-1}{D}},
\end{equation*}
and
\begin{equation*}
\left| \frac{\partial f_{t_{1}}^{t_{2}}}{\partial x_{i}}\right|
=\left| \frac{\partial f}{\partial x_{i}}\right| \chi _{\left\{
t_{1}<\left| f\right| \leq t_{2}\right\} }
\end{equation*}
we get
\begin{equation}
\left( t_{2}-t_{1}\right) \left| \left\{ x\in \mathbb{R}^{n}:\left|
f(x)\right| \geq t_{2}\right\} \right| ^{1-1/D}\leq
nC\prod_{i=1}^{n}\left( \int_{\left\{ t_{1}<\left| f\right| \leq
t_{2}\right\} }\left| f_{x_{i}}\right| d\mu \right)
^{\frac{A_{i}+1}{D}}.  \label{uno}
\end{equation}
Using that $\left| f_{x_{i}}\right| \leq \left| \nabla f\right| $, (\ref{uno}%
) implies
\begin{equation*}
\left( t_{2}-t_{1}\right) \left| \left\{ x\in \mathbb{R}^{n}:\left|
f(x)\right| \geq t_{2}\right\} \right| ^{1-1/D}\leq nC\int_{\left\{
t_{1}<\left| f\right| \leq t_{2}\right\} }\left| \nabla f\right|
\,d\mu
\end{equation*}
and from this inequality the locally absolutely continuity of
$f_{\mu }^{\ast }$ follows easily using the same argument as in
\cite[page 137] {MMAdv}.

Let $s>0$ and $h>0,$ pick $t_{1}=f_{\mu }^{\ast }(s+h),$
$t_{2}=f_{\mu }^{\ast }(s),$ then by (\ref{uno}) we get
\begin{equation*}
\left( f_{\mu }^{\ast }(s)-f_{\mu }^{\ast }(s+h)\right)
s^{1-1/D}\leq nC\prod_{i=1}^{n}\left( \int_{\left\{ f_{\mu }^{\ast
}(s+h)<\left| f\right|
\leq f_{\mu }^{\ast }(s)\right\} }\left| f_{x_{i}}\right| d\mu \right) ^{%
\frac{A_{i}+1}{D}}.
\end{equation*}
Thus,
\begin{equation*}
\frac{\left( f_{\mu }^{\ast }(s)-f_{\mu }^{\ast }(s+h)\right) }{h}%
s^{1-1/D}\leq nC\prod_{i=1}^{n}\left( \frac{1}{h}\int_{\left\{
f_{\mu }^{\ast }(s+h)<\left| f\right| \leq f_{\mu }^{\ast
}(s)\right\} }\left| f_{x_{i}}\right| d\mu \right)
^{\frac{A_{i}+1}{D}}.
\end{equation*}
Letting $h\rightarrow 0$ we obtain (\ref{tres}).

$iii) \Rightarrow iv)$

Let $1\leq p<\infty .$ For $0<s<t,$ we get
\begin{align}
O_{\mu }(f,t)& =\frac{1}{t}\int_{0}^{t}\left( f_{\mu }^{\ast
}(s)-f_{\mu
}^{\ast }(t)\right) ds  \notag \\
& =\frac{1}{t}\int_{0}^{t}\left( \int_{s}^{t}\left( -f_{\mu }^{\ast
}\right)
^{\prime }(\tau )\,d\tau \right) ds  \notag \\
& =\frac{1}{t}\int_{0}^{t}s\left( -f_{\mu }^{\ast }\right) ^{^{\prime }}(s)ds%
\text{ }  \notag \\
& \leq \frac{t^{1/D}}{t}\int_{0}^{t}s^{1-1/D}\left( -f_{\mu }^{\ast
}\right)
^{^{\prime }}(s)ds  \notag \\
\!\!& \leq nC\frac{t^{1/D}}{t}\int_{0}^{t}\prod_{i=1}^{n}\left( \frac{d}{%
d\tau }\int_{\left\{ \left| f\right| >f_{\mu }^{\ast }(\tau
)\right\}
}\left| f_{x_{i}}\right| d\mu \right) ^{\frac{A_{i}+1}{D}}ds\text{ \ (by (%
\ref{tres}))}  \notag \\
\!\!& \!\!\leq nC\frac{t^{1/D}}{t}\int_{0}^{t}\left(
\prod_{i=1}^{n}\left[ \frac{d}{d\tau }\int_{\left\{ \left| f\right|
>f_{\mu }^{\ast }(\tau )\right\} }\left| f_{x_{i}}\right| d\mu
\right] ^{\frac{A_{i}+1}{D}}\right)
^{\ast }(s)ds\text{ \ {(by \eqref{HL})}}  \notag \\
& \leq nCt^{1/D}\left( \frac{1}{t}\int_{0}^{t}\left(
\prod_{i=1}^{n}\left[ \frac{d}{d\tau }\int_{\left\{ \left| f\right|
>f_{\mu }^{\ast }(\tau )\right\} }\left| f_{x_{i}}\right| d\mu
\right] ^{p\frac{A_{i}+1}{D}}\right) ^{\ast }(s)ds\right)
^{1/p}\text{ (by H\"{o}lder).}  \notag
\end{align}
Thus
\begin{equation}
\left( t^{-1/D}O_{\mu }(f,t)\right) ^{p}\preceq \frac{1}{t}%
\int_{0}^{t}\left( \prod_{i=1}^{n}\left[ \frac{d}{d\tau
}\int_{\left\{ \left| f\right| >f_{\mu }^{\ast }(\tau )\right\}
}\left| f_{x_{i}}\right| d\mu \right] ^{p\frac{A_{i}+1}{D}}\right)
^{\ast }(s)ds.  \label{arem}
\end{equation}
Notice that the previous computation shows that
\begin{equation}
O_{\mu }(f,t)=\frac{1}{t}\int_{0}^{t}s\left( -f_{\mu }^{\ast
}\right) ^{^{\prime }}(s)\,ds.  \label{1111}
\end{equation}
If $p=1$, then{\ using \eqref{1111} and \eqref{arem} we get}
\begin{align*}
\int_{0}^{t}O_{\mu }(f,s)s^{-1/D}ds&
=\int_{0}^{t}\frac{s^{-1/D}}{s}\left(
\int_{0}^{s}z\left( -f_{\mu }^{\ast }\right) ^{^{\prime }}(z)dz\right) ds \\
& =\int_{0}^{t}z\left( -f_{\mu }^{\ast }\right) ^{^{\prime
}}(z)\left(
\int_{z}^{t}\frac{s^{-1/D}}{s}ds\right) dz \\
& \leq \int_{0}^{t}z\left( -f_{\mu }^{\ast }\right) ^{^{\prime
}}(z)\left(
\int_{z}^{\infty }\frac{s^{-1/D}}{s}ds\right) dz \\
& =D\int_{0}^{t}z^{1-1/D}\left( -f_{\mu }^{\ast }\right) ^{^{\prime
}}(z)dz
\\
& \leq nCD\int_{0}^{t}\left( \prod_{i=1}^{n}\left[ \frac{d}{d\tau }%
\int_{\left\{ \left| f\right| >f_{\mu }^{\ast }(\cdot )\right\}
}\left| f_{x_{i}}(x)\right| d\mu (x)\right]
^{\frac{A_{i}+1}{D}}\right) ^{\ast }(s)\,ds.
\end{align*}
If $1<p<\infty ,$ then u{sing (\ref{arem}), Hardy's Inequalities
(see \cite[page 124]{BS}) and (\ref{1111}) we get}
\begin{align*}
\int_{0}^{t}\left( O_{\mu }(f,s)s^{-1/D}\right) ^{p}ds&
=\int_{0}^{t}\left( \frac{s^{-1/D}}{s}\int_{0}^{s}z\left( -f_{\mu
}^{\ast }\right) ^{^{\prime
}}(z)dz\right) ^{p}ds \\
& \leq \int_{0}^{t}\left( \frac{1}{s}\int_{0}^{s}z^{1-1/D}\left(
-f_{\mu
}^{\ast }\right) ^{^{\prime }}(z)dz\right) ^{p}ds= \\
& \preceq \int_{0}^{t}\left( z^{1-1/D}\left( -f_{\mu }^{\ast
}\right)
^{^{\prime }}(z)\right) ^{p}dz \\
& \preceq \int_{0}^{t}\left( \prod_{i=1}^{n}\left[ \frac{d}{d\tau }%
\int_{\left\{ \left| f\right| >f_{\mu }^{\ast }(\cdot )\right\}
}\left| f_{x_{i}}(x)\right| d\mu (x)\right]
^{p\frac{A_{i}+1}{D}}\right) ^{\ast }(s)\,ds.
\end{align*}
By Lemma \ref{01} and \cite[Exercise 10, page 88]{BS}, we get
\begin{eqnarray*}
\int_{0}^{t}\left( O_{\mu }(f,\cdot )^{p}\left( \cdot \right) ^{-\frac{p}{D}%
}\,\right) ^{\ast }(s)\,ds &\preceq &\int_{0}^{t}\left(
\prod_{i=1}^{n}\left[ \frac{d}{d\tau }\int_{\left\{ \left| f\right|
>f_{\mu }^{\ast }(\cdot
)\right\} }\left| f_{x_{i}}(x)\right| d\mu (x)\right] ^{p\frac{A_{i}+1}{D}%
}\right) ^{\ast }(s)\,ds \\
&\leq &\int_{0}^{t}\prod_{i=1}^{n}\left( \left[ \frac{d}{d\tau }%
\int_{\left\{ \left| f\right| >f_{\mu }^{\ast }(\cdot )\right\}
}\left| f_{x_{i}}(x)\right| d\mu (x)\right]
^{p\frac{A_{i}+1}{D}}\right) ^{\ast }(s)\,ds
\end{eqnarray*}

$iv)\Rightarrow v)$

By (\ref{Hardy}) and (\ref{holder X})$,$ we obtain
\begin{align}
\left\| O_{\mu }(f,s)s^{-\frac{1}{D}}\right\| _{L^{1}}& \leq
4nCD\left\| \prod_{i=1}^{n}\left( \left[ \frac{d}{d\tau
}\int_{\left\{ \left| f\right|
>f_{\mu }^{\ast }(\tau )\right\} }\left| f_{x_{i}}\right| d\mu \right] ^{%
\frac{A_{i}+1}{D}}\right) ^{\ast }(s)\right\| _{L^{1}}  \label{a01} \\
& \leq 4nCD\prod_{i=1}^{n}\left\| \left( \frac{d}{d\tau
}\int_{\left\{ \left| f\right| >f_{\mu }^{\ast }(\tau )\right\}
}\left| f_{x_{i}}\right| d\mu \right) ^{\ast }(s)\right\|
_{L^{1}}^{\frac{A_{i}+1}{D}}.  \notag
\end{align}
Moreover we get
\begin{eqnarray}  \label{55}
\left\| \left( \frac{d}{d\tau }\int_{\left\{ \left| f\right| >f_{\mu
}^{\ast }(\tau )\right\} }\left| f_{x_{i}}\right| d\mu \right)
^{\ast }(s)\right\| _{L^{1}} &=&\int_{0}^{\infty }\left(
\frac{d}{d\tau }\int_{\left\{ \left| f\right| >f_{\mu }^{\ast }(\tau
)\right\} }\left| f_{x_{i}}\right| d\mu
\right) ^{\ast }(s)\, ds  \notag \\
&=&\int_{0}^{\infty }\frac{d}{d\tau }\left( \int_{\left\{ \left|
f\right|
>f_{\mu }^{\ast }(\tau )\right\} }\left| f_{x_{i}}\right| d\mu \right) d\tau
\notag \\
&=&\int_{\mathbb{R}^{n}}\left| f_{x_{i}}\right| d\mu .
\end{eqnarray}
%

Taking into account that $\frac{\partial }{\partial t}f_{\mu }^{\ast
\ast }(t)=-(f_{\mu }^{\ast \ast }(t)-f_{\mu }^{\ast }(t))/t$ and
$f_{\mu }^{\ast \ast }(\infty )=0$ (since $f\in
W_{0}^{1,1}(\mathbb{R}^{n},\mu ))$ by the Fundamental Theorem
Calculus we have
\begin{equation*}
f_{\mu }^{\ast \ast }(t)=\int_{t}^{\infty }\left( f_{\mu }^{\ast
\ast }(\tau )-f_{\mu }^{\ast }(\tau )\right) \,\frac{d\tau }{\tau }.
\end{equation*}
Therefore
\begin{align*}
\Vert f\Vert _{L^{\frac{D}{D-1},1}}:=\int_{0}^{\infty }f_{\mu
}^{\ast \ast }(s)s^{-\frac{1}{D}}\,ds& =\int_{0}^{\infty }\left[
\int_{s}^{\infty }\left(
f_{\mu }^{\ast \ast }(\tau )-f_{\mu }^{\ast }(\tau )\right) \,\frac{d\tau }{%
\tau }\right] s^{-\frac{1}{D}}\,ds \\
& =\int_{0}^{\infty }\left( f_{\mu }^{\ast \ast }(\tau )-f_{\mu
}^{\ast
}(\tau )\right) \,\left[ \frac{1}{\tau }\int_{0}^{\tau }s^{-\frac{1}{D}}\,ds%
\right] \,d\tau \\
& =\frac{D}{D-1}\int_{0}^{\infty }\left( f_{\mu }^{\ast \ast }(\tau
)-f_{\mu }^{\ast }(\tau )\right) \tau ^{-\frac{1}{D}}\,d\tau .
\end{align*}
{Conclusion follows by (\ref{a01}) and \eqref{55}.}

$v)\Rightarrow i)$

It is consequence of the continuous embedding of Lorentz space $L_{\mu }^{%
\frac{D}{D-1},1}$ into Lebesgue space $L_{\mu }^{\frac{D}{D-1}}$ and
the relation between weighted arithmetic and geometric mean.
\end{proof}

\begin{remark}
As in the classical case we have that \eqref{inq 1} implies a better
inequality involving the Lorentz norm $\Vert \cdot\Vert _{L^{\frac{D}{D-1}%
,1}}$(see \textit{e.g.} \cite{C_optimal}, \cite{C2018}, \cite{MMP07} and the
bibliography therein).
\end{remark}

\begin{remark}
{By (\ref{arem}) we get }

\begin{equation}  \label{99}
t^{-\frac{1}{D}}O_{\mu }(f,t)\preceq \frac{1}{t}\int_{0}^{t}\prod_{i=1}^{n}%
\left( \frac{d}{d\tau }\int_{\left\{ \left| f\right| >f_{\mu }^{\ast }(\tau
)\right\} }\left| f_{x_{i}}\right| d\mu \right) ^{\frac{A_{i}+1}{D}}ds:=I(t),
\end{equation}
that is a pointwise oscillation inequality in multiplicative form. Moreover
by \eqref{99} we recover the classical oscillation inequality (see \cite
{MMP07} and \cite{MMAdv})
\begin{equation*}
t^{-\frac{1}{D}}O_{\mu }(f,t)\preceq \left| \nabla f\right| ^{\ast \ast }(t).
\end{equation*}
As matter of the fact by H\"{o}lder inequality
\begin{eqnarray}  \label{77}
I(t) &\leq &\prod_{i=1}^{n}\left( \frac{1}{t}\int_{0}^{t}\left( \frac{d}{ds}%
\int_{\{|f|>f_{\mu }^{\ast }(s)\}}\left| f_{x_{i}}\right| d\mu \right)
ds\right) ^{\frac{A_{i}+1}{D}}  \\
&=&\prod_{i=1}^{n}\left( \frac{1}{t}\int_{\{|f|>f_{\mu }^{\ast }(t)\}}\left|
f_{x_{i}}\right| d\mu \right) ^{\frac{A_{i}+1}{D}}  \notag \\
&\leq &\prod_{i=1}^{n}\left( \frac{1}{t}\int_{0}^{t}\left| f_{x_{i}}\right|
_{\mu }^{\ast }(s)ds\right) ^{\frac{A_{i}+1}{D}} \text{{\ (by \eqref{HL})}}
\notag \\
&\preceq &\left| \nabla f\right| ^{\ast \ast }(t). \notag
\end{eqnarray}
\end{remark}

\section{Anisotropic inequalities in rearrangement invariant spaces\label%
{applica}}

%
{In this section starting from the oscillation inequality \eqref{O} we
derive some anisotropic inequalities in $\mathbb{R}^n$ in the general
setting of rearrangement invariant spaces. }

We will use throughout this Section the following notation
\begin{equation*}
(\tilde{f}_{x_{i}})_{\mu}^{*}(t):=\left( \frac{d}{ds}\int_{\{|f|>f_{\mu
}^{\ast }(s)\}}\left| f_{x_{i}}\right| d\mu \right) ^{\ast }(t).
\end{equation*}

In order to prove these kind of results we need the following results.

\begin{lemma}
\label{pesos}Let $1<p<\infty $ and let $v$ be a weight (a positive locally
integrable function on $\left( 0,\infty \right) ).$ Let
\begin{equation*}
u(t)=\frac{\partial }{\partial t}\left( 1+\int_{t}^{1}\frac{v^{\frac{-1}{p-1}%
}(s)}{s^{\frac{p}{p-1}}}ds\right) ^{1-p}.
\end{equation*}
Then, there exists $C>0$ such that
\begin{equation*}
\left( \int_{0}^{1}f^{\ast \ast }(s)^{p}u(s)ds\right) ^{1/p}\leq C\left(
\int_{0}^{1}\left( O_{\mu }(f,s)\right) ^{p}v(s)ds\right) ^{1/p}+\left(
\int_{0}^{1}u(s)ds\right) ^{1/p}\int_{0}^{1}f^{\ast }(t)ds.
\end{equation*}
\end{lemma}

\begin{proof}
Since
\begin{eqnarray*}
\left( \int_{0}^{t}u(s)ds\right) ^{1/p}\left( \int_{t}^{1}\frac{v(s)^{\frac{%
-1}{p-1}}}{s^{\frac{p}{p-1}}}ds\right) ^{\left( p-1\right) /p} &\leq
&\left(
1+\int_{t}^{1}\frac{v^{\frac{-1}{p-1}}}{s^{\frac{p}{p-1}}}ds\right)
^{(1-p)/p}\left( \int_{t}^{1}\frac{v(s)^{\frac{-1}{p-1}}}{s^{\frac{p}{p-1}}}%
ds\right) ^{\left( p-1\right) /p} \\
&\leq &1,
\end{eqnarray*}
the result follows from \cite[Lemma 5.4]{BMR}.
\end{proof}

\begin{lemma}
\label{rakoto}(see \cite{rakotoson} and \cite{AT}) Let $f\in W_{0}^{1,1}(%
\mathbb{R}^{n},\mu ),$ then
\begin{equation*}
\int_{0}^{t}(\tilde{f}_{x_{i}})_{\mu}^{*}(\tau )d\tau \leq
\int_{0}^{t}\left| f_{x_{i}}\right| _{\mu }^{\ast }(\tau )d\tau ,\text{ \ \ (%
}t\geq 0)
\end{equation*}
therefore by (\ref{Hardy}) for any r.i space $X$ on $(\mathbb{R}^{n},\mu )$
we have that
\begin{align*}
\left\| (\tilde{f}_{x_{i}})_{\mu}^{*}\right\| _{\bar{X}}& \leq \left\|
\left| f_{x_{i}}\right| _{\mu }^{\ast }\right\| _{\bar{X}} =\left\|
f_{x_{i}}\right\| _{X}.
\end{align*}
\end{lemma}

\subsection{Convexification of r.i. spaces}

\subsubsection{The $X^{(q)}$ convexification}

\ Let $X$ be a r.i. space on $\mathbb{R}^{n}$ and $X^{(q)}$ its $q-$%
con\-ve\-xi\-fication defined in (\ref{convessifi}). In the next
theorem we state some anisotropic inequalities for functions $f$
such that $f_{x_{i}}\in X^{(p_{i})} $ with $p_i\geq 1$ for
$i=1,\cdots,n$.

\begin{theorem}
\label{th2} Let $X$ be a r.i. space on $(\mathbb{R}^{n},\mu)$ and $f\in{\
W_{0}^{1,1}(\mathbb{R}^{n},\mu)}$. If $p_{1},\cdots ,p_{n}\geq 1$, then
\begin{equation}
\Vert t^{-1/D}[f_{\mu }^{\ast \ast }(t)-f_{\mu }^{\ast }(t)]\Vert _{\bar{X}%
^{(\overline{p})}}\preceq\prod_{i=1}^{n}\left\| f_{x_{i}}\right\|
_{X^{(p_{i})}}^{\frac{A_{i}+1}{D}},  \label{Xq}
\end{equation}
where $\overline{p}$ and $D$ are defined in \eqref{harm} and in \eqref{D},
respectively {and the involved norms are defined as in (\ref{convessifi}).
Moreover }

\noindent i) if $\underline{\alpha}_{X}>\frac{\overline{p}}{D}$, then
\begin{equation}
\Vert t^{-1/D}f_{\mu }^{\ast \ast }(t)\Vert_{\bar{X}^{(\overline{p}
)}}\preceq\,\prod_{i=1}^{n}\left\|f_{x_{i}}\right\|_{X^{(p_{i})}}^{\frac{
A_{i}+1}{D}};  \label{emb1}
\end{equation}

\noindent ii) if $\bar{\alpha}_{X}<\frac{\overline{p}}{D}$, then
\begin{equation}
\Vert f\Vert _{L^{\infty }}\preceq\,\prod_{i=1}^{n}\left\|
f_{x_{i}}\right\|_{X^{(p_{i})}}^{\frac{A_{i}+1}{D}}+\left\|f\right\|
_{L^1_{\mu}+L^{\infty }}.  \label{emb2}
\end{equation}
\end{theorem}

\begin{proof}
Since $\bar{X}^{(\overline{p})}$ is a r.i space, by Theorem
\ref{main th} part $iv)$ we get
\begin{align*}
\left\| \left( f_{\mu }^{\ast \ast }(s)-f_{\mu }^{\ast }(s)\right) s^{-\frac{%
1}{D}}\right\| _{\bar{X}^{(\overline{p})}}& \preceq \left\| \prod_{i=1}^{n}%
\left[ (\tilde{f}_{x_{i}})_{\mu }^{\ast }\right] ^{\frac{A_{i}+1}{D}%
}\right\| _{\bar{X}^{(\overline{p})}} \\
& =\left\| \prod_{i=1}^{n}\left[ (\tilde{f}_{x_{i}})_{\mu }^{\ast }\right] ^{%
\frac{p_{i}}{\bar{p}}\left(
\frac{\bar{p}}{p_{i}}\frac{A_{i}+1}{D}\right)
}\right\| _{\bar{X}^{(\overline{p})}} \\
& \leq \prod_{i=1}^{n}\left\| (\tilde{f}_{x_{i}})_{\mu }^{\ast
}(t)\right\|
_{\bar{X}^{(p_{i})}}^{\frac{A_{i}+1}{D}}\text{ (by (\ref{holder X}))} \\
& \leq \prod_{i=1}^{n}\left\| f_{x_{i}}\right\| _{X^{(p_{i})}}^{\frac{A_{i}+1%
}{D}}\text{ \ (by Lemma \ref{rakoto}).}
\end{align*}

\noindent On the other hand, since (see \cite[Proposition 5.13,
Chapter 3 ] {BS})
\begin{equation*}
\underline{\alpha }_{\bar{X}^{(\overline{p})}}=\frac{\underline{\alpha }_{X}%
}{\bar{p}}\text{ and }\bar{\alpha}_{\bar{X}^{(\overline{p})}}=\frac{\bar{%
\alpha}_{X}}{\bar{p}}
\end{equation*}
(\ref{emb1}) and (\ref{emb2}) follows from Theorem \ref{teoind}.
\end{proof}

\begin{remark}
We stress in the previous proof if we start from \eqref{77} instead of %
\eqref{O} we can not consider the case when $p_1=\cdots=p_n=1$.
\end{remark}

\medskip In the particular case $X=L^1_\mu$, the previous Theorem can be
detailed as in the following proposition.

\begin{proposition}
\label{prop L1} Let $p_{1},\cdots ,p_{n}\geq 1$ and $f\in {W_{0}^{1,1}(%
\mathbb{R}^{n},\mu)}$.

\noindent i) If {\ $\overline{p}<D$}, then
\begin{equation}
\Vert f\Vert _{L^{{\overline{p}}^{\ast },\overline{p}}_\mu
}\preceq\,\prod_{i=1}^{n}\left\| f_{x_{i}}\right\| _{L^{p_{i}}_\mu }^{\frac{%
A_{i}+1}{D}},  \label{p1}
\end{equation}
where $\overline{p}$ is defined as in (\ref{harm}) and ${\overline{p}}^{\ast
}=\frac{\overline{p}D}{D-\overline{p}}$.

{\ \noindent ii) If $\overline{p}=D$, then
\begin{equation}
\left( \int_{0}^{1}\left( \frac{f_{\mu }^{\ast \ast }(t)}{1+\ln \frac{1}{t}}%
\right) ^{D}\frac{dt}{t}\right) ^{1/D}\preceq\,\prod_{i=1}^{n}\left\|
f_{x_{i}}\right\| _{L^{p_{i}}_\mu}^{\frac{A_{i}+1}{D}}+\left\| f\right\|
_{L^1_\mu+L^{\infty }}.  \label{p2}
\end{equation}
}

\noindent iii) {If} $\overline{p}>D$, then
\begin{equation}
\Vert f\Vert _{L^{\infty }}\preceq\,\prod_{i=1}^{n}\left\| f_{x_{i}}\right\|
_{L^{p_{i}}_\mu }^{\frac{A_{i}+1}{D}}+\left\| f\right\| _{L^1_\mu+L^{\infty
}}.  \label{p3}
\end{equation}
\end{proposition}

\begin{proof}
Since $\left( L^1_\mu\right) ^{(\bar{p})}=L^{\bar{p}}_\mu$ and $\underline{%
\alpha }_{L^{\bar{p}}}=\bar{\alpha}_{L^{\bar{p}}}=${$\frac{1}{{\bar{p}}}$, (%
\ref{p1}) and (\ref{p3}) follows from Theorem \ref{th2}. {We have
only to
prove } (\ref{p2}).} {If $v(t)=\frac{1}{t}$ in Lemma \ref{pesos}, then $%
u(t)=(D-1)\left( \frac{1}{1+\ln \left( \frac{1}{t}\right) }\right) ^{-D}%
\frac{1}{s}$. Under this choice Lemma \ref{pesos} allows us to get }
\begin{equation*}
\left( \int_{0}^{1}\left( \frac{f_{\mu }^{\ast \ast }(t)}{1+\ln \left( \frac{%
1}{t}\right) }\right) ^{D}\frac{dt}{t}\right) ^{1/D}\preceq \left(
\int_{0}^{1}\left( f_{\mu }^{\ast \ast }(t)-f_{\mu }^{\ast }(t)\right) ^{D}%
\frac{dt}{t}\right) ^{1/D}+f_{\mu }^{\ast \ast }(1).
\end{equation*}
\end{proof}

\begin{remark}
Our result implies Theorem \ref{Theorem:Sobolev W1p}. Indeed it gives an
embedding in Lorentz spaces. If $A_{1}=\cdots =A_{n}=0$ we obtain the same
results of \cite{Tar} and \cite{KP} for what concerns \eqref{p1}.
\end{remark}

\begin{remark}
Let $\Omega \subset \mathbb{R}^{n}$ be a bounded domain and $f\in
C_{c}^{1}\left( \Omega \right) $. Arguing as in the proof of Proposition \ref
{prop L1}, we get
\begin{equation*}
\left( \int_{0}^{\mu \left( \Omega \right) }\left( \frac{f_{\mu }^{\ast \ast
}(t)}{1+\ln \frac{\mu \left( \Omega \right) }{t}}\right) ^{D}\frac{dt}{t}%
\right) ^{1/D}\preceq \,\prod_{i=1}^{n}\left\| f_{x_{i}}\right\| _{L_{\mu
}^{p_{i}}}^{\frac{A_{i}+1}{D}}+\frac{1}{\mu \left( \Omega \right) }%
\int_{\Omega }\left| f\right| d\mu .
\end{equation*}
Since
\begin{equation*}
\sup_{0<t<\mu \left( \Omega \right) }\frac{f_{\mu }^{\ast \ast }(t)}{\left(
1+\ln \frac{\mu \left( \Omega \right) }{t}\right) ^{\frac{D-1}{D}}}\preceq
\left( \int_{0}^{\mu \left( \Omega \right) }\left( \frac{f_{\mu }^{\ast \ast
}(t)}{1+\ln \frac{\mu \left( \Omega \right) }{t}}\right) ^{D}\frac{dt}{t}%
\right) ^{1/D},
\end{equation*}
we obtain the following anisotropic Trudinger inequality
\begin{equation*}
\sup_{0<t<\mu \left( \Omega \right) }\frac{f_{\mu }^{\ast \ast }(t)}{\left(
1+\ln \frac{\mu \left( \Omega \right) }{t}\right) ^{\frac{D-1}{D}}}\preceq
\,\prod_{i=1}^{n}\left\| f_{x_{i}}\right\| _{L_{\mu }^{p_{i}}}^{\frac{A_{i}+1%
}{D}}+\frac{1}{\mu \left( \Omega \right) }\int_{\Omega }\left| f\right| d\mu
.
\end{equation*}
\end{remark}

\subsubsection{The $X^{\langle q\rangle}$ convexification}

\begin{theorem}
Let $X$ be a r.i. space on $(\mathbb{R}^{n},\mu )$ and $f\in {W_{0}^{1,1}(%
\mathbb{R}^{n},\mu )}$. If $p_{1},\cdots ,p_{n}\geq 1$, then
\begin{equation}
\Vert t^{-1/D}[f_{\mu }^{\ast \ast }(t)-f_{\mu }^{\ast }(t)]\Vert _{\bar{X}%
^{\left\langle \overline{p}\right\rangle }}\preceq \prod_{i=1}^{n}\left\|
f_{x_{i}}\right\| _{X^{\left\langle p_{i}\right\rangle }}^{\frac{A_{i}+1}{D}%
},  \notag
\end{equation}
where $\overline{p}$ and $D$ are defined in \eqref{harm} and in
\eqref{D}, respectively and the involved norms are defined as in
\eqref{convessifi1}. Moreover

\noindent i) if $\underline{\alpha }_{\bar{X}^{\left\langle \overline{p}%
\right\rangle }}>\frac{1}{D}$, then
\begin{equation*}
\Vert t^{-1/D}f_{\mu }^{\ast \ast }(t)\Vert _{\bar{X}^{\left\langle
\overline{p}\right\rangle }}\preceq \,\prod_{i=1}^{n}\left\|
f_{x_{i}}\right\| _{X^{\left\langle p_{i}\right\rangle }}^{\frac{A_{i}+1}{D}%
};
\end{equation*}

\noindent ii) if $\bar{\alpha}_{\bar{X}^{\left\langle \overline{p}%
\right\rangle }}<\frac{1}{D}$, then
\begin{equation*}
\Vert f\Vert _{L^{\infty }}\preceq \,\prod_{i=1}^{n}\left\|
f_{x_{i}}\right\| _{X^{\left\langle p_{i}\right\rangle }}^{\frac{A_{i}+1}{D}%
}+\left\| f\right\| _{L_{\mu }^{1}+L^{\infty }}.
\end{equation*}
\end{theorem}

\begin{proof}
By Theorem \ref{main th} part $iv)$ with $p=\bar{p},$ we get
\begin{eqnarray*}
\frac{1}{t}\int_{0}^{t}\left( O_{\mu }(f,\cdot )^{\bar{p}}\left(
\cdot
\right) ^{-\frac{\bar{p}}{D}}\,\right) ^{\ast }(s)ds &\preceq &\frac{1}{t}%
\int_{0}^{t}\prod_{i=1}^{n}\left[ (\tilde{f}_{x_{i}})_{\mu }^{\ast }(\tau )%
\right] ^{\bar{p}\left( \frac{A_{i}+1}{D}\right) }d\tau  \\
&=&\frac{1}{t}\int_{0}^{t}\prod_{i=1}^{n}\left[ \left( (\tilde{f}%
_{x_{i}})_{\mu }^{\ast }(\tau )\right) ^{p_{i}}\right] ^{\frac{\bar{p}}{p_{i}%
}\left( \frac{A_{i}+1}{D}\right) }d\tau  \\
&\leq &\prod_{i=1}^{n}\left( \frac{1}{t}\int_{0}^{t}\left[ \left( (\tilde{f}%
_{x_{i}})_{\mu }^{\ast }(\tau )\right) ^{p_{i}}\right] d\tau \right) ^{\frac{%
\bar{p}}{p_{i}}\left( \frac{A_{i}+1}{D}\right) }
\end{eqnarray*}
Let $X$ a r.i. space, then
\begin{align*}
\Vert t^{-1/D}[f_{\mu }^{\ast \ast }(t)-f_{\mu }^{\ast }(t)]\Vert _{\bar{X}%
^{\left\langle \overline{p}\right\rangle }}& =\left\| \left( \frac{1}{t}%
\int_{0}^{t}\left( O_{\mu }(f,\cdot )^{\bar{p}}\left( \cdot \right) ^{-\frac{%
\bar{p}}{D}}\,\right) ^{\ast }(s)ds\right) ^{1/\bar{p}}\right\| _{\bar{X}} \\
& \preceq \left\| \prod_{i=1}^{n}\left(
\frac{1}{t}\int_{0}^{t}\left[ \left( (\tilde{f}_{x_{i}})_{\mu
}^{\ast }(\tau )\right) ^{p_{i}}\right] d\tau \right)
^{\frac{1}{p_{i}}\left( \frac{A_{i}+1}{D}\right) }\right\|
_{\bar{X}}
\\
& \leq \prod_{i=1}^{n}\left\| \left( \frac{1}{t}\int_{0}^{t}\left[ \left( (%
\tilde{f}_{x_{i}})_{\mu }^{\ast }(\tau )\right) ^{p_{i}}\right]
d\tau
\right) ^{\frac{1}{p_{i}}}\right\| _{\bar{X}}^{\left( \frac{A_{i}+1}{D}%
\right) } \\
& =\prod_{i=1}^{n}\left\| (\tilde{f}_{x_{i}})_{\mu }^{\ast }(\tau
)\right\|
_{\bar{X}^{\left\langle p_{i}\right\rangle }}^{\left( \frac{A_{i}+1}{D}%
\right) } \\
& \leq \prod_{i=1}^{n}\left\| f_{x_{i}}\right\|
_{\bar{X}^{\left\langle p_{i}\right\rangle
}}^{\frac{A_{i}+1}{D}}\text{ \ (by Lemma \ref{rakoto}).}
\end{align*}

The statements (\ref{emb1}) and (\ref{emb2}) follows from Theorem
\ref{teoind}.
\end{proof}

\subsection{Generalized Lorentz spaces}

{In this section we state some anisotropic inequalities for
functions $f$ such that the partial derivatives are in some
generalized Lorentz spaces. More precisely as a consequence of
Theorem \ref{th2} we obtain the following theorem.}

\begin{theorem}
\label{aqqq} Let $w\in B_{\min (p_{1},\cdots ,p_{n})}$, $f\in {W_{0}^{1,1}(%
\mathbb{R}^{n},\mu)}$, $p_{1},\cdots ,p_{n}\geq 1$, $q_{1},\cdots
,q_{n}\geq 1$ and the involved norms are defined as in
\eqref{lornorm}.

\noindent i) If $\underline{\alpha }_{\Lambda ^{{\overline{p}},\overline{q}%
}(w)}>\frac{1}{D}$, then
\begin{equation*}
\Vert f\Vert _{\Lambda _{\mu }^{{\overline{p}}^{\ast },\overline{q}%
}(w)}\preceq\,\prod_{i=1}^{n}\left\| f_{x_{i}}\right\| _{\Lambda _{\mu
}^{p_{i},q_{i}}(w)}^{\frac{A_{i}+1}{D}},
\end{equation*}
where $\bar{p}$ and $\bar{q}$ are defined as in (\ref{harm}) and ${\overline{%
p}}^{\ast }=\frac{\overline{p}D}{D-\overline{p}}$.

\noindent ii) If $\bar{\alpha}_{\Lambda ^{{\overline{p}},\overline{q}}(w)}>%
\frac{1}{D}$, then
\begin{equation*}
\Vert f\Vert _{L^{\infty }}\preceq\,\prod_{i=1}^{n}\left\| f_{x_{i}}\right\|
_{\Lambda _{\mu }^{p_{i},q_{i}}(w)}^{\frac{A_{i}+1}{D}}+\left\| f\right\|
_{L^1_\mu+L^{\infty }}.
\end{equation*}
\noindent iii) In the remaining cases we get
\begin{equation*}
\left( \int_{0}^{1}f^{\ast \ast }(s)^{\bar{q}}u(s)ds\right) ^{1/\bar{q}%
}\preceq\,\prod_{i=1}^{n}\left\| f_{x_{i}}\right\| _{\Lambda _{\mu
}^{p_{i},q_{i}}(w)}^{\frac{A_{i}+1}{D}}+\left\| f\right\|
_{L^1_\mu+L^{\infty }},
\end{equation*}
where $u(t)=\frac{\partial }{\partial t}\left( 1+\int_{t}^{1}\frac{\left( s^{%
\frac{\bar{q}}{\bar{p}}-\frac{\bar{q}}{D}-1}w(s)\right) ^{\frac{-1}{\bar{q}-1%
}}}{s^{\frac{\bar{q}}{\bar{q}-1}}}ds\right) ^{1-\bar{q}}$.
\end{theorem}

\begin{proof}
Since $w\in B_{\min (p_{1},\cdots ,p_{n})}$ all the spaces $\Lambda
^{p_{i},q_{i}}(w)$ are r.i spaces and from $\min (p_{1},\cdots ,p_{n})\leq {%
\overline{p}}$ it follows that $\Lambda ^{{\overline{p}},\overline{q}}(w)$
is a r.i. space. We note that
\begin{equation}
1=\frac{\bar{q}}{D}\sum_{i=1}^{n}\frac{A_{i}+1}{q_{i}}.
\label{holddd}
\end{equation}
By Theorem \ref{main th} we obtain
\begin{eqnarray*}
\Vert t^{-1/D}[f_{\mu }^{\ast \ast }(t)-f_{\mu }^{\ast }(t)]\Vert _{\Lambda
^{{\bar{p}},\overline{q}}(w)} &\preceq &\left\| \prod_{i=1}^{n}\tilde{f}%
_{x_{i}}(\tau )^{\frac{A_{i}+1}{D}}\right\| _{\Lambda ^{{\overline{p}},%
\overline{q}}(w)} \\
&=&\left( \int_{0}^{\infty }\left( t^{\frac{1}{{\overline{p}}}-\frac{1}{\bar{%
q}}}\prod_{i=1}^{n}\tilde{f}_{x_{i}}(t)^{\frac{A_{i}+1}{D}}\right) ^{\bar{q}%
}w(t)\,dt\right) ^{\frac{1}{\bar{q}}} \\
&\leq &\prod_{i=1}^{n}\left( \int_{0}^{\infty }\left( t^{\frac{1}{{p}_{i}}-%
\frac{1}{q_{i}}}\tilde{f}_{x_{i}}(t)^{\frac{A_{i}+1}{D}}\right)
^{q_{i}}w(t)\,dt\right) ^{\frac{A_{i}+1}{q_{i}D}}\text{(by (\ref{holddd}) and (%
\ref{holder X}))} \\
&=&\prod_{i=1}^{n}\left\| \tilde{f}_{x_{i}}\right\| _{\Lambda
^{p_{i},q_{i}}(w)}^{\frac{A_{i}+1}{D}} \\
&\leq &\prod_{i=1}^{n}\left\| f_{x_{i}}\right\| _{\Lambda _{\mu
}^{p_{i},q_{i}}(w)}^{\frac{A_{i}+1}{D}}\text{ \ (by Lemma \ref{rakoto}).}
\end{eqnarray*}
Part \textit{i}) follows form \textit{i}) of Theorem \ref{teoind}, because
\begin{eqnarray*}
\Vert t^{-1/D}[f_{\mu }^{\ast \ast }(t)-f_{\mu }^{\ast }(t)]\Vert _{\Lambda
^{{\overline{p}},\overline{q}}(w)} &\simeq &\Vert t^{-1/D}f_{\mu }^{\ast
\ast }(t)\Vert _{\Lambda ^{{\overline{p}},\overline{q}}(w)} \\
&\simeq &\Vert f_{\mu }^{\ast \ast }(t)\Vert _{\Lambda ^{{\overline{p}}%
^{\ast },\overline{q}}(w)} \\
&\simeq &\Vert f\Vert _{\Lambda _{\mu }^{{\overline{p}}^{\ast },\overline{q}%
}(w)}.
\end{eqnarray*}

Part \textit{ii}) is consequence of \textit{i}) of  Theorem \ref{teoind}.

Let us prove now \textit{iii}). If we denote $I:=\Vert t^{-1/D}[f_{\mu }^{\ast
\ast }(t)-f_{\mu }^{\ast }(t)]\Vert _{\Lambda ^{{\overline{p}},\overline{q}%
}(w)}$ we get
\begin{eqnarray*}
I &=&\left( \int_{0}^{\infty }\left( s^{\frac{1}{\bar{p}}-\frac{1}{\bar{q}}%
}\left( t^{-1/D}[f_{\mu }^{\ast \ast }(t)-f_{\mu }^{\ast }(t)]\right) ^{\ast
}(s)\right) ^{\bar{q}}w(s)\,ds\right) ^{1/\bar{q}} \\
&\simeq &\left( \int_{0}^{\infty }\left( s^{\frac{1}{\bar{p}}-\frac{1}{\bar{q%
}}}\frac{1}{s}\int_{0}^{s}\left( t^{-1/D}[f_{\mu }^{\ast \ast }(t)-f_{\mu
}^{\ast }(t)]\right) ^{\ast }(z)\,dz\right) ^{\bar{q}}w(s)\,ds\right) ^{1/\bar{q}%
}\text{ \ \ \ (since }w\in B_{\bar{p}}) \\
&\geq &\left( \int_{0}^{\infty }\left( s^{\frac{1}{\bar{p}}-\frac{1}{\bar{q}}%
}\frac{1}{2s}\int_{0}^{2s}\left( t^{-1/D}[f_{\mu }^{\ast \ast }(t)-f_{\mu
}^{\ast }(t)]\right) ^{\ast }(z)\,dz\right) ^{\bar{q}}w(s)\,ds\right) ^{1/\bar{q}%
} \\
&\geq &\left( \int_{0}^{\infty }\left( s^{\frac{1}{\bar{p}}-\frac{1}{\bar{q}}%
}\frac{1}{2s}\int_{0}^{2s}\left( z^{-1/D}[f_{\mu }^{\ast \ast }(z)-f_{\mu
}^{\ast }(z)]\right) dz\right) ^{\bar{q}}w(s)\,ds\right) ^{1/\bar{q}} \\
&\geq &\left( \int_{0}^{\infty }\left( s^{\frac{1}{\bar{p}}-\frac{1}{\bar{q}}%
}\frac{1}{2s}\int_{s}^{2s}\left( z^{-1/D}[f_{\mu }^{\ast \ast }(z)-f_{\mu
}^{\ast }(z)]\right) dz\right) ^{\bar{q}}w(s)\,ds\right) ^{1/\bar{q}} \\
&\geq &\left( \int_{0}^{\infty }\left( s^{\frac{1}{\bar{p}}-\frac{1}{\bar{q}}%
}\frac{s[f_{\mu }^{\ast \ast }(s)-f_{\mu }^{\ast }(s)]}{2s}%
\int_{s}^{2s}\left( z^{-1/D-1}\right) dz\right) ^{\bar{q}}w(s)\,ds\right) ^{1/%
\bar{q}}\text{ \ \ (by (\ref{crece}))} \\
&\simeq &\left( \int_{0}^{\infty }\left( s^{\frac{1}{\bar{p}}-\frac{1}{\bar{q%
}}-\frac{1}{D}}[f_{\mu }^{\ast \ast }(s)-f_{\mu }^{\ast }(s)]\right) ^{\bar{q%
}}w(s)\,ds\right) ^{1/\bar{q}} \\
&=&\left( \int_{0}^{\infty }[f_{\mu }^{\ast \ast }(s)-f_{\mu }^{\ast }(s)]^{%
\bar{q}}s^{\frac{\bar{q}}{\bar{p}}-\frac{\bar{q}}{D}-1}w(s)\,ds\right) ^{1/%
\bar{q}}.
\end{eqnarray*}
Considering as weights
\begin{equation*}
v(s):=s^{\frac{\bar{q}}{\bar{p}}-\frac{\bar{q}}{D}-1}w(s)\text{ \ and }u(t)=%
\frac{\partial }{\partial t}\left( 1+\int_{t}^{1}\frac{\left( s^{\frac{\bar{q%
}}{\bar{p}}-\frac{\bar{q}}{D}-1}w(s)\right) ^{\frac{-1}{\bar{q}-1}}}{s^{%
\frac{\bar{q}}{\bar{q}-1}}}ds\right) ^{1-\bar{q}},
\end{equation*}
the result follows from Lemma \ref{pesos}.
\end{proof}

{The following corollaries follow from the previous theorem considering $w=1$
and $w(t)=(1+|\ln t|)^\alpha$ with $\alpha\in\mathbb{R}$, respectively, and
recalling that}
\begin{equation*}
\bar{\alpha}_{L^{{p,q}}(\log L)^{\alpha }}=\underline{\alpha }_{L^{{p,q}%
}(\log L)^{\alpha }}=\bar{\alpha}_{L^{{p,q}}}=\underline{\alpha }_{L^{{p,q}%
}}=\frac{1}{p}.
\end{equation*}

\begin{corollary}
\label{C1} Let $f\in C_{c}^{1}(\mathbb{R}^{n})$, $p_{1},\cdots ,p_{n}\geq 1$
and $q_{1},\cdots ,q_{n}\geq 1$.

\noindent i) If $\bar{p}<D$, then
\begin{equation*}
\Vert f\Vert _{L_{\mu }^{{\overline{p}}^{\ast },\overline{q}%
}}\preceq\,\prod_{i=1}^{n}\left\| f_{x_{i}}\right\| _{L_{\mu
}^{p_{i},q_{i}}}^{\frac{A_{i}+1}{D}},
\end{equation*}
where $\bar{p}$ and $\bar{q}$ are defined as in (\ref{harm}) and ${\overline{%
p}}^{\ast }=\frac{\overline{p}D}{D-\overline{p}}$.

\noindent ii) If $\bar{p}>D$, then
\begin{equation*}
\Vert f\Vert _{L^{\infty }(\mathbb{R}^{n})}\preceq\,\prod_{i=1}^{n}\left\|
f_{x_{i}}\right\| _{L_{\mu }^{p_{i},q_{i}}}^{\frac{A_{i}+1}{D}}+\left\|
f\right\| _{L^1_\mu+L^{\infty }}.
\end{equation*}
\noindent iii) If $\bar{p}=D$, then
\begin{equation*}
\left( \int_{0}^{1}\left( \frac{f_{\mu }^{\ast \ast }(s)}{1+\ln \frac{1}{s}}%
\right) ^{\bar{q}}\frac{ds}{s}\right) ^{1/\bar{q}}\preceq\,\prod_{i=1}^{n}%
\left\| f_{x_{i}}\right\| _{L_{\mu }^{p_{i},q_{i}}}^{\frac{A_{i}+1}{D}%
}+\left\| f\right\| _{L^1_\mu+L^{\infty }}.
\end{equation*}
\end{corollary}

\begin{corollary}
Let $f\in C_{c}^{1}(\mathbb{R}^{n})$, $p_{1},\cdots ,p_{n}\geq 1$ , $%
q_{1},\cdots ,q_{n}\geq 1$ and $\alpha \in \mathbb{R}$.

\noindent i) If $\bar{p}<D$, then
\begin{equation*}
\Vert f\Vert _{L_{\mu }^{{\overline{p}}^{\ast },\overline{q}}(\log
L)^{\alpha }}\preceq\,\prod_{i=1}^{n}\left\| f_{x_{i}}\right\| _{L_{\mu
}^{p_{i},q_{i}}(\log L)^{\alpha }}^{\frac{A_{i}+1}{D}}.
\end{equation*}
where $\bar{p}$ and $\bar{q}$ are defined as in (\ref{harm}) and ${\overline{%
p}}^{\ast }=\frac{\overline{p}D}{D-\overline{p}}$.

\noindent ii) If $\bar{p}>D$, then
\begin{equation*}
\Vert f\Vert _{L^{\infty }(\mathbb{R}^{n})}\preceq\,\prod_{i=1}^{n}\left\|
f_{x_{i}}\right\| _{L_{\mu }^{p_{i},q_{i}}(\log L)^{\alpha }}^{\frac{A_{i}+1%
}{D}}+\left\| f\right\| _{L^1\mu+L^{\infty }}.
\end{equation*}
\noindent iii) If $\bar{p}=D$, then
\begin{equation*}
\left( \int_{0}^{1}\left( \frac{f_{\mu }^{\ast \ast }(s)}{1+\ln \frac{1}{s}}%
\right) ^{\bar{q}}\left( 1+\ln \frac{1}{s}\right) ^{\alpha }\frac{ds}{s}%
\right) ^{1/\bar{q}}\preceq\,\prod_{i=1}^{n}\left\| f_{x_{i}}\right\|
_{L_{\mu }^{p_{i},q_{i}}(\log L)^{\alpha }}^{\frac{A_{i}+1}{D}}+\left\|
f\right\| _{L^1_\mu+L^{\infty }}.
\end{equation*}
\end{corollary}

{\noindent When $A_1=\cdots=A_n=0$, \textit{i}) of Corollary
\ref{C1} is contained in \cite{Tar} and \cite{KP}. For our knowledge
the other ones are new. }

\subsection{The Gamma spaces}

\begin{theorem}
Let $w$ be a weight satisfying condition (\ref{gweg}), $f\in {W_{0}^{1,1}(%
\mathbb{R}^{n},\mu )}$, $p_{1},\cdots ,p_{n}\geq 1$ and the involved
norms are defined as in \eqref{gammanorm}.

\noindent i) If $\underline{\alpha }_{\Gamma ^{{\overline{p}}^{\ast },}(w)}>%
\frac{1}{D}$, then
\begin{equation*}
\Vert t^{-1/D}f_{\mu }^{\ast }(t)\Vert _{\Gamma ^{{\overline{p}}^{\ast
}}(w)}\preceq \prod_{i=1}^{n}\left\| f_{x_{i}}\right\| _{\Gamma _{\mu
}^{p_{i}}(w)}^{\frac{A_{i}+1}{D}},
\end{equation*}
where $\bar{p}$ is defined as in (\ref{harm}) and ${\overline{p}}^{\ast }=%
\frac{\overline{p}D}{D-\overline{p}}$.

\noindent ii) If $\bar{\alpha}_{\Gamma ^{{\overline{p}}^{\ast },}(w)}<\frac{1%
}{D}$, then
\begin{equation*}
\Vert f\Vert _{L^{\infty }}\preceq \prod_{i=1}^{n}\left\| f_{x_{i}}\right\|
_{\Gamma _{\mu }^{p_{i}}(w)}^{\frac{A_{i}+1}{D}}+\left\| f\right\| _{L_{\mu
}^{1}+L^{\infty }}.
\end{equation*}
\noindent iii) In the remaining cases we get
\begin{equation*}
\left( \int_{0}^{1}f^{\ast \ast }(s)^{{\overline{p}}^{\ast }}u(s)ds\right)
^{1/{\overline{p}}^{\ast }}\preceq \prod_{i=1}^{n}\left\| f_{x_{i}}\right\|
_{\Gamma _{\mu }^{p_{i}}(w)}^{\frac{A_{i}+1}{D}}+\left\| f\right\| _{L_{\mu
}^{1}+L^{\infty }},
\end{equation*}
where $u(t)=\frac{\partial }{\partial t}\left( 1+\int_{t}^{1}\frac{\left(
s^{-\frac{{\overline{p}}^{\ast }}{D}}w(s) \right) ^{\frac{-1}{{\overline{p}}%
^{\ast }-1}}}{s^{\frac{{\overline{p}}^{\ast }}{{\overline{p}}^{\ast }-1}}}%
ds\right) ^{1-{\overline{p}}^{\ast }}$.
\end{theorem}

\begin{proof}
\begin{eqnarray*}
\Vert t^{-1/D}[f_{\mu }^{\ast \ast }(t)-f_{\mu }^{\ast }(t)]\Vert _{\Gamma ^{%
{\overline{p}}^{\ast }}(w)} &\preceq &\left\| \prod_{i=1}^{n}(\tilde{f}%
_{x_{i}})_{\mu }^{\ast }(t)^{\frac{A_{i}+1}{D}}\right\| _{\Gamma ^{{%
\overline{p}}^{\ast },}(w)} \\
&=&\left( \int_{0}^{\infty }\left( \frac{1}{t}\int_{0}^{t}\prod_{i=1}^{n}(%
\tilde{f}_{x_{i}})_{\mu }^{\ast }(s)^{\frac{A_{i}+1}{D}}ds\right) ^{{%
\overline{p}}^{\ast }}w(t)\,dt\right) ^{\frac{1}{{\overline{p}}^{\ast }}} \\
&\leq &\prod_{i=1}^{n}\left( \int_{0}^{\infty }\left( \frac{1}{t}%
\int_{0}^{t}(\tilde{f}_{x_{i}})_{\mu }^{\ast }(s)ds\right)
^{p_{i}}w(t)\,dt\right) ^{\frac{A_{i}+1}{p_{i}D}}\text{(by
(\ref{holddd})
and (\ref{holder X}))} \\
&=&\prod_{i=1}^{n}\left\| (\tilde{f}_{x_{i}})_{\mu }^{\ast }\right\|
_{\Gamma _{\mu }^{p_{i}}(w)}^{\frac{A_{i}+1}{D}} \\
&\leq &\prod_{i=1}^{n}\left\| f_{x_{i}}\right\| _{\Gamma _{\mu
}^{p_{i}}(w)}^{\frac{A_{i}+1}{D}}\text{ \ (by Lemma \ref{rakoto}).}
\end{eqnarray*}
Part \textit{i}) follows form \textit{i}) of Theorem \ref{teoind},
since
\begin{equation*}
\Vert t^{-1/D}[f_{\mu }^{\ast \ast }(t)-f_{\mu }^{\ast }(t)]\Vert _{\Gamma ^{%
{\overline{p}}^{\ast }}(w)}\simeq \Vert t^{-1/D}f_{\mu }^{\ast \ast
}(t)\Vert _{\Gamma ^{{\overline{p}}^{\ast }}(w)}\simeq \Vert
t^{-1/D}f_{\mu }^{\ast }(t)\Vert _{\Gamma ^{{\overline{p}}^{\ast
}}(w)}\text{.}
\end{equation*}

Part \textit{ii}) is consequence of \textit{i}) of Theorem
\ref{teoind}.

Let us prove now \textit{iii}). If we denote $I:=\Vert
t^{-1/D}[f_{\mu }^{\ast \ast }(t)-f_{\mu }^{\ast }(t)]\Vert _{\Gamma
^{{\overline{p}}^{\ast }}(w)}$ using the same argument that in part
iii) of Theorem \ref{aqqq} we easily obtain
\begin{equation*}
I\geq \left( \int_{0}^{\infty }[f_{\mu }^{\ast \ast }(s)-f_{\mu
}^{\ast
}(s)]^{\bar{q}}s^{-\frac{{\overline{p}}^{\ast }}{D}}w(s)\,ds\right) ^{1/{%
\overline{p}}^{\ast }}.
\end{equation*}
Considering as weights
\begin{equation*}
v(s):=s^{-\frac{{\overline{p}}^{\ast }}{D}}w(s)\text{ \ and }u(t)=\frac{%
\partial }{\partial t}\left( 1+\int_{t}^{1}\frac{\left( s^{-\frac{{\overline{%
p}}^{\ast }}{D}}w(s)\right) ^{\frac{-1}{{\overline{p}}^{\ast }-1}}}{s^{%
\frac{{\overline{p}}^{\ast }}{{\overline{p}}^{\ast }-1}}}ds\right) ^{1-{%
\overline{p}}^{\ast }},
\end{equation*}
the result follows from Lemma \ref{pesos}.
\end{proof}

\subsection{The $G\Gamma(m,p,w)$-spaces}

\begin{theorem}
Let $w$ be a satisfying (\ref{ggweg}), $f\in
{W_{0}^{1,1}(\mathbb{R}^{n},\mu )}$, $p_{1},\cdots ,p_{n}\geq 1$ and
the involved norms are defined as in (\ref{ggammanorm}).

\noindent i) If $\underline{\alpha }_{G\Gamma ({\overline{p}}^{\ast },m,w)}>%
\frac{1}{D}$, then
\begin{equation*}
\Vert t^{-1/D}f_{\mu }^{\ast }(t)\Vert _{G\Gamma ({\overline{p}}^{\ast
},m,w)}\preceq \prod_{i=1}^{n}\left\| f_{x_{i}}\right\| _{G\Gamma ({p}%
_{i},m,w)}^{\frac{A_{i}+1}{D}},
\end{equation*}
where $\bar{p}$ is defined as in (\ref{harm}) and ${\overline{p}}^{\ast }=%
\frac{\overline{p}D}{D-\overline{p}}$.

\noindent ii) If $\bar{\alpha}_{G\Gamma ({\overline{p}}^{\ast },m,w)}<\frac{1%
}{D}$, then
\begin{equation*}
\Vert f\Vert _{L^{\infty }}\preceq \prod_{i=1}^{n}\left\| f_{x_{i}}\right\|
_{G\Gamma ({p}_{i},m,w)}^{\frac{A_{i}+1}{D}}+\left\| f\right\| _{L_{\mu
}^{1}+L^{\infty }}.
\end{equation*}
\noindent iii) In the remaining cases we get
\begin{equation*}
\left( \int_{0}^{1}f^{\ast \ast }(s)^{{m}}u(s)ds\right) ^{1/{m}}\preceq
\prod_{i=1}^{n}\left\| f_{x_{i}}\right\| _{G\Gamma ({p}_{i},m,w)}^{\frac{%
A_{i}+1}{D}}+\left\| f\right\| _{L_{\mu }^{1}+L^{\infty }},
\end{equation*}
where $u(t)=\frac{\partial }{\partial t}\left( 1+\int_{t}^{1}\frac{\left(
s^{-\frac{m}{D}}w(s)\right) ^{\frac{-1}{{m}-1}}}{s^{\frac{{m}}{{m}-1}}}%
ds\right) ^{1-{m}}$.
\end{theorem}

\begin{proof}
By Theorem \ref{main th} we get
\begin{eqnarray*}
\Vert t^{-1/D}[f_{\mu }^{\ast \ast }(t)-f_{\mu }^{\ast }(t)]\Vert _{G\Gamma (%
{\overline{p}}^{\ast },m,w)} &\preceq &\left\| \prod_{i=1}^{n}(\tilde{f}%
_{x_{i}})_{\mu }^{\ast }(t)^{\frac{A_{i}+1}{D}}\right\| _{G\Gamma ({%
\overline{p}}^{\ast },m,w)} \\
&=&\left( \int_{0}^{\infty }\left( \int_{0}^{t}\prod_{i=1}^{n}(\tilde{f}%
_{x_{i}})_{\mu }^{\ast }(s)^{{\overline{p}}^{\ast }\frac{A_{i}+1}{D}%
}ds\right) ^{m/{\overline{p}}^{\ast }}w(t)\,dt\right) ^{\frac{1}{{m}}} \\
&\leq &\left( \int_{0}^{\infty }\prod_{i=1}^{n}\left( \int_{0}^{t}(\tilde{f}%
_{x_{i}})_{\mu }^{\ast }(s)^{p_{i}}ds\right) ^{\frac{{m}}{p_{i}}\frac{A_{i}+1%
}{D}}w(t)\,dt\right) ^{\frac{1}{{m}}} \\
&\leq &\prod_{i=1}^{n}\left( \int_{0}^{\infty }\left( \left( \int_{0}^{t}(%
\tilde{f}_{x_{i}})_{\mu }^{\ast }(s)^{p_{i}}ds\right) ^{\frac{{m}}{p_{i}}%
}w(t)\,\right) dt\right) ^{\frac{1}{{m}}\frac{A_{i}+1}{D}} \\
&=&\prod_{i=1}^{n}\left\| (\tilde{f}_{x_{i}})_{\mu }^{\ast }\right\|
_{G\Gamma ({p}_{i},m,w)}^{_{\frac{A_{i}+1}{D}}} \\
&\leq &\prod_{i=1}^{n}\left\| f_{x_{i}}\right\| _{G\Gamma ({p}_{i},m,w)}^{%
\frac{A_{i}+1}{D}}\text{ \ (by Lemma \ref{rakoto}).}
\end{eqnarray*}
Part \textit{i}) follows form \textit{i}) of Theorem \ref{teoind},
because
\begin{equation*}
\Vert t^{-1/D}[f_{\mu }^{\ast \ast }(t)-f_{\mu }^{\ast }(t)]\Vert _{G\Gamma (%
{\overline{p}}^{\ast },m,w)}\simeq \Vert t^{-1/D}f_{\mu }^{\ast \ast
}(t)\Vert _{G\Gamma ({\overline{p}}^{\ast },m,w)}
\end{equation*}

Part \textit{ii}) is consequence of \textit{i}) of Theorem
\ref{teoind}.

Let us prove now \textit{iii}). If we denote $I:=\Vert
t^{-1/D}[f_{\mu }^{\ast \ast }(t)-f_{\mu }^{\ast }(t)]\Vert
_{G\Gamma ({\overline{p}}^{\ast },m,w)}$ using the method of part
iii) of Theorem \ref{aqqq} we get
\begin{equation*}
I\geq \left( \int_{0}^{\infty }[f_{\mu }^{\ast \ast }(s)-f_{\mu
}^{\ast }(s)]^{m}s^{-\frac{m}{D}}w(s)\,ds\right) ^{1/{m}}.
\end{equation*}
Considering as weights
\begin{equation*}
v(s):=s^{-\frac{m}{D}}w(s)\text{ \ and }u(t)=\frac{\partial }{\partial t}%
\left( 1+\int_{t}^{1}\frac{\left( s^{-\frac{m}{D}}w(s)\right) ^{\frac{-1}{{m}%
-1}}}{s^{\frac{{m}}{{m}-1}}}ds\right) ^{1-{m}},
\end{equation*}
the result follows from Lemma \ref{pesos}.
\end{proof}

\bigskip

\textit{Acknowledgments. } {The first author has been partially supported by
FFABR and the second one has been partially supported in part by Grants
MTM2016-77635-P, MTM2016-75196-P (MINECO) and by 2017SGR358.}

The second author gratefully acknowledges the generous hospitality
of the Universit\`{a} degli Studi di Napoli \textquotedblleft
Federico II\textquotedblright \, Dipartimento di Matematica
\textquotedblleft R. Caccioppoli\textquotedblright \, where much of
this work was done.

\end{document}